\RequirePackage{fix-cm}
\documentclass[smallextended]{svjour3}       
\smartqed  
\usepackage{graphicx}
\usepackage{amsfonts}
\usepackage{epsfig}

\begin{document}

\title{The analysis of FETI-DP preconditioner for \\
full DG discretization of elliptic problems} 


\titlerunning{FETI-DP for full DG}        

\author{Maksymilian Dryja      \and Juan Galvis \and  Marcus  Sarkis 
}

\authorrunning{Dryja, Galvis \& Sarkis} 

\institute{M. Dryja \at
Department of Mathematics, Warsaw University, Banacha 2, 00-097 Warsaw, Poland.
This research was supported 
in part by the Polish Sciences Foundation under grant 2011/01/B/ST1/01179.\\
               \email{ dryja@mimuw.edu.pl}           
           \and
           J. Galvis \at
Departamento de Matem{\' a}ticas, Universidad Nacional de Colombia, 
Bogot{\'a}, Colombia.\\
\email{jcgalvisa@unal.edu.co}
\and
M. Sarkis \at Department of Mathematical
Sciences at Worcester Polytechnic
Institute, 100 Institute Road, Worcester, MA 01609, USA, and 
Instituto Nacional de Matem\'{a}tica Pura
e Aplicada (IMPA), Estrada Dona Castorina 110,
CEP 22460-320, Rio de Janeiro, Brazil. \\
\email{msarkis@wpi.edu}
}

\date{Received: date / Accepted: date}

\maketitle

\begin{abstract}
In this paper a discretization 
based on discontinuous Galerkin (DG) method for 
an elliptic two-dimensional problem with discontinuous
coefficients is considered. The problem is posed on a polygonal 
region $\Omega$ which is a union of $N$ disjoint polygonal subdomains
$\Omega_i$
of diameter $O(H_i)$. The discontinuities of the coefficients, possibly
very large, are assumed to occur only across the subdomain interfaces
$\partial \Omega_i$.  In each $\Omega_i$  a conforming quasiuniform
triangulation 
with parameters $h_i$ is constructed. We assume that the resulting triangulation in $\Omega$ is also 
conforming, i.e., the meshes are assumed to match across the subdomain
interfaces.   On the fine triangulation the problem is discretized by
a DG method. For solving 
the resulting discrete system, a FETI-DP type method is proposed  
and analyzed. It is established that the condition number of the 
preconditioned linear system 
is estimated by 
$C(1 + \max_i \log H_i/h_i)^2$ with a constant $C$ independent of 
$h_i$, $H_i$ and the jumps of coefficients. The method is well suited for 
parallel computations and it can be extended to three-dimensional problems. 
This result is an extension, to the case of full fine-grid DG
discretization,  of the previous result [SIAM J. Numer. Anal., 51
(2013), pp.~400--422] where it was considered a conforming finite
element method inside the subdomains and  a discontinuous Galerkin 
method only across the subdomain interfaces.  Numerical results 
are presented to validate the theory. 

\end{abstract}


\section{Introduction}
In this paper we consider a boundary value problem for elliptic second order partial differential equations
with highly discontinuous coefficients and homogeneous Dirichlet boundary condition.  
The problem is 
posed  on  a polygonal region $\Omega$ which is a  union of disjoint 
two-dimensional polygonal subdomains $\Omega_i$ of diameter $O(H_i)$.  We assume 
that this partition $\{\Omega_i\}_{i=1}^N$ is geometrically conforming,
i.e., for all $i$ and $j$ with $ i \neq j$, the intersection 
$\partial\Omega_i \cap \partial \Omega_j$ is either empty, a common corner
or a common edge of $\Omega_i$ and $\Omega_j$. We consider the case where the discontinuities  
of the coefficients 
are assumed to occur only across $\partial \Omega_i$. The problem is 
approximated by the symmetric interior penalty discontinuous Galerkin  (SIPDG) method inside
each $\Omega_i$, with $h_i$ as a mesh parameter.  The meshes are 
assumed to match across the interfaces $\partial \Omega_i\cap \partial \Omega_j$.  
In order to deal with the nonconformity of the DG spaces across
$\partial \Omega_i$, a discrete problem is further formulated  
using the SIPDG method on the $\partial \Omega_i$; 
see \cite{MR664882,MR1885715,MR2837483,MR2372235,MR2431403}. 
The main goal of this paper is to develop 
the FETI-DP methodology for this DG discretization. We show that 
the designed FETI-DP method is almost optimal with rate of convergence 
independent of the coefficients jumps. The analysis of the discussed 
precondicioner is based on the analysis used in \cite{MR3033016}.
For that we construct  refinement of the fine meshes on each
$\Omega_i$ and introduce special interpolation operators
which allow to switch from spaces of piecewise linear functions defined on the fine mesh on $\Omega_i$
to spaces of piecewise linear and continuous functions defined on the
refinement of the fine mesh on $\Omega_i$ and vise-versa. The
interpolation operators used in this paper are similar to those
introduced 
\cite{MR1297465,MR1469678} and also used in \cite{MR2970761}.  Properties of the interpolations operators are proved in this paper.
In this paper we extend our results published in \cite{MR3033016} to the full 
DG discretization of the considered problem. In \cite{MR3033016} a FETI-DP preconditioner 
was designed and analyzed for the problem discretized by composed  
finite element  and DG methods, i.e., 
by continuous Finite Element Method (FEM) inside each $\Omega_i$ and the symmetric interior penalty DG method on 
the interfaces $\partial\Omega_i$ only, while here in this paper 
a FETI-DP preconditioner for the case full DG discretization of the problem 
is designed and analyzed. The proposed FETI-DP algorithm is essentially
algebraic and does not require any interpolation operator for the design
of the algorithm, therefore, it can be naturally extended to three-dimentional
problems and for high-order and discretizations elasticity, Stokes and Maxwell. 
An interpolation operator is only
required for the analysis, and to the best of our knowledge it is the
first complete analysis ever developped for any Neumann-Neumann type of 
discretization such as FETI, FETI-DP, BDD, BDDC and NN, and for any 
of the classical interior penalty DG discretizations considered in 
\cite{MR664882,MR1885715,MR2837483}. We expect that the analysis 
developped here can be extended to more general problems as the one
mentioned above. We also remark that a new technique based on 
two interpolation operators were introduced to avoid a condition
on the number of elements that can touch a corner of 
a substructure interface,therefore, the algorithm developed here 
can applied also for subdomains generated from graph partitioners. Furthermore, 
we believe that such technique can be extended to other types of 
full DG discretizations beyond \cite{MR664882,MR1885715,MR2837483}.

For the developing of FETI-DP methods for the continuous FEM, see 
the Introduction of \cite{MR3033016} and the references therein. See also
\cite{MR1813746,MR1802366,MR2421044,MR1921914,MR1835470}.

 In this paper, the full DG discrete problem is reduced to 
the Schur complement problem with respect to 
unknowns on the interfaces of the subdomains $\Omega_i$. For that, discrete 
harmonic functions defined in a special way, i.e., in the DG sense, 
are used. We note that there are 
unknowns on both sides of the interfaces $\partial \Omega_i$. 
This means that the unknowns on both sides of the interfaces
should be kept as degrees of freedom of the linear Schur complement system to be solved.
 These issues
characterize some of the main difficulties on 
designing and analyzing FETI-DP type methods 
for full DG discretizations. Distinctively from the 
classical conforming FEM discretizations, here a  
double layer of Lagrange multipliers are needed on 
interfaces rather than a single layer of Lagrange multipliers
as normally is seen in FETI-DP for conforming FEMs.
Despite the fact we follow the FETI-DP 
abstract approach, 
which was aimed to single layer of Lagrange multipliers, 
see for example \cite{MR2104179},
in this paper we successfully overcome this difficulty.

The algorithm we develop in this paper is as follows. 
Let ${\Gamma}_i^\prime$ be
the union of all edges $\bar{E}_{ij}$ and $\bar{E}_{ji}$ 
which are common to $\Omega_i$ and $\Omega_j$,
where  $\bar{E}_{ij}$ and  $\bar{E}_{ji}$ 
refer to the $\Omega_i$ and $\Omega_j$ sides, respectively, 
see Figure \ref{fig:dof}. 
We note that each ${\Gamma}_i^\prime$ has interface 
unknowns (degrees of freedom) corresponding to nodal points  which are the endpoint of edges of fine triangulation belonging to 
$\overline{\partial \Omega_i\setminus \partial \Omega}$ and
$\overline{E_{ji}} \subset \partial \Omega_j$. Unknowns corresponding to vertices of fine triangles which intersect 
$\bar{E}_{ij} \subset \partial \Omega_i$  
and $\bar{E}_{ji} \subset \partial \Omega_j$  by only one vertex 
are treated as interior unknowns. We now need to couple $\Gamma_i^\prime$
with the other side of the interface $\Gamma_j^\prime$. We first impose continuity at  the  
interface unknowns at the corners of $\Gamma_i^\prime$ (which are 
corners of $\Gamma_i$ and common 
endpoints of $\bar{E}_{ji}$) with the interface unknowns at corners of the $\Gamma_j^\prime$,
see Figure \ref{fig:contV}. These unknowns are called primal. 
The remaining interface unknowns on $\Gamma_i^\prime$ and 
$\Gamma_j^\prime$ are called dual 
and have jumps, hence, Lagrange multipliers are introduced to eliminate 
these jumps, see Figure \ref{fig:cont}. 
 For the dual system with Lagrange multipliers, 
a  special block diagonal preconditioner is 
designed. It leads to independent local problems on $\Gamma_i^\prime$ 
for $1\leq i \leq N.$ It is proved that the proposed method 
is almost optimal with a condition number estimate bounded by 
$C(1 + \max_i\log H_i/h_i)^2$, where $C$ does not depend  
on $h_i$, $H_i$, the number of subdomains $\Omega_i$  and 
the jumps in the coefficients.

 The method can be 
extended to full DG discretizations of three-dimensional problems.

The paper is organized as follows. In Section 2 the differential 
problem and a full DG discretization are formulated. In Section 3,  
the Schur complement problem is derived using discrete harmonic 
functions defined in a special way (in the DG sense). In Section 4,
the  so-called FETI-DP method is introduced, i.e., the Schur 
complement problem is reformulated by imposing continuity 
for the primal variables and by using Lagrange multipliers at the 
dual variables, and a special block diagonal preconditioner is defined. The 
main results of the paper are  Theorem \ref{teorema32} and 
Lemma  \ref{lemma45}. Section 5 is devoted to some technical tools 
and auxiliary lemmas to analyze the FETI-DP preconditioner. The 
proofs of these results are given in the Appendix A and B.
In Section 6 
numerical tests
are reported which confirm the theoretical results.

\section{Differential and discrete problems}
In this section we discuss the continuous problem 
and its DG discretization. The resulting 
discrete problem is taken into consideration for preconditioning.

\subsection{Differential problem}
Consider the following problem: \emph{Find $u_{ex}^* \in H^1_0(\Omega)$ such that}
\begin{equation} \label{contproblem} 
a(u^*_{ex}, v) = f(v) \quad \mbox{ for all } v \in H^1_0(\Omega),
\end{equation}
where $\rho_i > 0$ is a constant, $f \in L^2(\Omega)$ and
$$
a(u, v) := \sum^N_{i=1} \int_{\Omega_i} \rho_i(x) \nabla u \cdot \nabla 
v \,dx \quad \mbox{and} \quad f(v) := \int_\Omega fv\,dx.
$$

We assume that $\overline{\Omega} = \cup^N_{i=1} \overline{\Omega}_i$ and 
the substructures $\Omega_i$ are disjoint shaped regular polygonal 
subregions of diameter $O(H_i)$. We assume that the partition 
$\{\Omega_i\}_{i=1}^N$  
is geometrically conforming,  i.e., for all $i$ and $j$ with  $ i \neq j$, the intersection 
$\partial\Omega_i \cap \partial \Omega_j$ is either empty, a common corner
or a common  edge of $\Omega_i$ and $\Omega_j$. For clarity we stress that here and below 
the identifier  \emph{edge} means a curve of continuous intervals and its two endpoints are called 
corners.  The collection of these corners on 
$\partial \Omega_i$ are referred as  the set of corners of $\Omega_i$.  Let us denote 
$\bar{E}_{ij} := \partial \Omega_i \cap \partial \Omega_j$ as an edge of
$\partial \Omega_i$ and $\bar{E}_{ji} := \partial \Omega_j \cap 
\partial \Omega_i$ as an edge of $\partial \Omega_j$. 
Sometimes we use the notation
$E_{ijh}$ and $E_{jih}$ to refer the sets of nodal 
points of the triangulation on $E_{ij}$ and $E_{ji}$ inherit from ${\cal{T}}_h^i$  and  ${\cal{T}}_h^j$, respectively, where the triangulations are defined below. Additionally we use the notation $\bar{E}_{ijh}$ 
and $\bar{E}_{jih}$ when the endpoints are included.    
Let us denote by ${\cal{J}}_H^{i,0}$ the set of indices $j$ such that  $\Omega_j$
 has a common edge $E_{ji}$ with $\Omega_i$. To take into account edges of
$\Omega_i$ which belong to the global boundary $\partial \Omega$, let us introduce a set of indices
${\cal{J}}_H^{i,\partial}$ to refer these edges. The set of indices of all
 edges of $\Omega_i$ is denoted by ${\cal{J}}_H^i = {\cal{J}}_H^{i,0} \cup \,
{\cal{J}}_H^{i,\partial}$. 

\subsection{Discrete problem} \label{discretization} 

Let us introduce a shape regular and quasiuniform 
triangulation (with 
triangular elements) ${\cal{T}}_h^i$ on $\Omega_i$ and let $h_i$ represent its mesh 
size. The resulting triangulation on $\Omega$ is 
matching across $\partial\Omega_i$. Let 
\[
X_{i}(\Omega_i):= \prod_{\tau \in {\cal{T}}_h^i} X_{\tau}
\]
 be the product space of finite element (FE) spaces $X_{\tau}$  which
consist of linear functions on the element $\tau$ belonging to  
${\cal{T}}_h^i$. We note that a function $u_i \in X_{i}(\Omega_i)$ 
allows discontinuities across elements of ${\cal{T}}_h^i$. We also
note that we do not assume that functions in $X_i(\Omega_i)$ vanish on
$\partial \Omega$. \\

   The global DG 
finite element space we consider is defined by 
\begin{equation} \label{defXi}
X(\Omega)= \prod_{i=1}^N X_i(\Omega_i) \equiv 
X_1(\Omega_1) \times  X_2(\Omega_2) \times \cdots \times  X_N(\Omega_N).
\end{equation}

We define ${\cal{E}}_h^{i,0}$ as the set of edges of the triangulation 
${\cal{T}}_h^i$ which are inside $\Omega_i$, and by
${\cal{E}}_h^{i,j}$, for $j \in 
{\cal{J}}_H^i$, the set of edges of the triangulation 
${\cal{T}}_h^i$ which are on $E_{ij}$. An edge 
$e \in {\cal{E}}_h^{i,0}$ is shared by two elements denoted by 
$\tau_+$ and $\tau_-$ of ${\cal{T}}_h^i$ with outward unit normal 
vectors ${\bf n}^+$ and ${\bf n}^-$, respectively, and denote
\[
\{\rho \nabla u\} = \frac{1}{2}(\rho_{\tau_{+}} \nabla u_{\tau_+} + 
\rho_{\tau_{-}} \nabla u_{\tau_-})~~~~\mbox{and}~~~~
[u] = u_{\tau_{+}} {\bf n}^+ +  u_{\tau_{-}} {\bf n}^-.
\]
 An edge 
$e \in {\cal{E}}_h^{i,\partial}$ is shared by one element  denoted by 
$\tau$ with outward unit normal 
vectors ${\bf n}$, and denote
\[
\{\rho\nabla u\} = \rho_\tau \nabla u_{\tau} ~~~~\mbox{and}~~~~
[u] = u_{\tau}{\bf n}.
\]

   The discrete problem we 
consider by the DG method is of the form:  \emph{Find 
$u^* = \{u_i^*\}_{i=1}^N \in X(\Omega)$ where $u_i^* \in X_i(\Omega_i)$, 
such that}
\begin{equation} \label{discproblem} 
a_h(u^*, v) = f(v) ~~~\;\; \mbox{ for all } v = \{v_i\}_{i=1}^N  \in X(\Omega),
\end{equation}
where the global bilinear from $a_h$ and the right hand side $f$ are assembled as
\begin{equation} \label{ah}
a_h(u, v) := \sum^N_{i=1} {a}^\prime_{i}(u, v) \;\; 
\mbox{and} \;\; f(v):= \sum^N_{i=1} \int_{\Omega_i} fv_i\,dx.
\end{equation}
Here, the local bilinear forms $a^\prime_i$, $i=1,\dots,N$, are defined as
\begin{equation} \label{aprimei}
{a}^\prime_{i}(u, v) := a_i(u_i, v_i) + 
s_{i,0}(u_i,v_i) + p_{i,0}(u,v) + 
s_{i,\partial}(u,v) + p_{i,\partial}(u,v)
\end{equation}
where $a_i$ is the local energy bilinear form,
\begin{equation} \label{bi}
  a_i(u_i, v_i) := \sum_{\tau \in {\cal{T}}_h^i} \int_{\tau} \rho_i \nabla u_i \cdot \nabla 
v_i \,dx.
\end{equation}
The (interior edges) symmetrized bilinear form $s_{i,0}$ is defined by
\begin{equation} \label{s0i}
 s_{i,0}(u_i, v_i):=  - \sum_{e \in  {\cal{E}}_h^{i,0}} \int_{e}
 \{\rho_i \nabla u_i\}\cdot [v_i] + \{\rho_i \nabla v_i\}\cdot [u_i]\,ds, 
\end{equation}
and the  (interior edges) penalty  bilinear form is given by
\begin{equation} \label{p0i}
 p_{i,0}(u,v): =  \sum_{e\in {\cal{E}}_h^{i,0}} 
\int_{e} \delta\frac{\rho_i}{h_e} [u_i].[v_i] \, ds.
\end{equation}
The corresponding symmetric and penalty form over the local interface edges are given by
\begin{equation} \label{spartiali}
s_{i,\partial}(u, v) := \sum_{ j \in {\cal{J}}_H^i}  \sum_{e \in  {\cal{E}}_h^{i,j}}
\int_{e}
\frac{1}{l_{ij}}\left( \rho_i\frac{\partial u_i}{\partial n}(v_j - v_i) + 
\rho_i\frac{\partial v_i}{\partial n}(u_j - u_i)\right)\,ds
\end{equation}
and
\begin{equation}\label{penaltypartiali}
p_{i,\partial}(u, v) :=  \sum_{ j \in {\cal{J}}_H^i} \sum_{e \in
  {\cal{E}}_h^{i,j}} \int_{e} \frac{\delta}{l_{ij}} 
\frac{\rho_i}{h_{e}} (u_i - u_j)(v_i-v_j)\,ds
\end{equation} 
respectively. Here and above, $h_e$ denotes the length of the edge $e$. When  $ j \in  {\cal{J}}_H^{i,0}$ we set 
$l_{ij} = 2$, while when  $j \in {\cal{J}}_H^{i,\partial}$ 
we denote the boundary edges $E_{ij}\subset \partial \Omega_i$ by 
${E}_{i\partial}$ and set $l_{i\partial} = 1$, and
on the artificial edge $E_{ji} \equiv E_{\partial i}$ we set  
$u_\partial =0$ and $v_\partial=0$. The partial derivative 
$\frac{\partial}{\partial n}$ denotes the 
outward normal derivative on $\partial \Omega_i$ and 
$\delta$ is the penalty positive parameter.\\

\begin{remark}
The discrete formulation used here is useful for 
the adequate formulation of our 
FETI-DP method. We note that (\ref{spartiali}) and
(\ref{penaltypartiali}) can also
be written in the same form as in  (\ref{s0i}) and (\ref{p0i}) without the term $\ell_{ij}$ (since now the
edges are counted once). We note also that others DG formulations for
discontinuous coefficients can also be considered
\cite{MR2837483,MR2002258}. We note that the design  
of FETI-DP methods for those formulations are the same, and
the analysis are simple modifications of the proofs we present here in
this paper. See for instance  \cite{MR2372024,MR2914789,MR2436092}
where a formulation based on harmonic averages of the coefficients is
studied. Some details for these case as well as numerical experiments 
will be presented elsewhere. We
note that three-dimensional versions can also be formulated and analyzed
by extending naturaly some ideas from this paper  and from
\cite{DDRennes}.  \\
\end{remark}

For $u = \{u_i\}_{i=1}^N, v = \{v_i\}_{i=1}^N \in X(\Omega)$, let us introduce the 
local positive bilinear forms 
\begin{equation} \label{formdi}
d_i(u, v) := a_i(u, v) + p_{i,0}(u, v) + p_{i,\partial}(u, v) 
\end{equation}
and the global positive bilinear form  assembled as
\begin{equation}
d_h(u, v) := \sum^N_{i=1} d_i(u, v).
\end{equation}
Note that the norm defined by $d_h(\cdot,\cdot)$ is a broken norm in 
$X(\Omega)$ with weights given by $\rho_i$, $\delta\frac{\rho_i}{h_e}$ and 
$\frac{\delta}{l_{ij}} \frac{\rho_i}{h_{e}}$. For 
$u = \{u_i\}^N_{i=1} \in X(\Omega)$,  this discrete norm is defined by 
$\parallel u \parallel^2_h  = d_h(u, u)$.  \\

It is known that there exists a $\delta_0 = O(1)>0$ and a positive constant
$c$, which does not depend on $\rho_i$, $H_i$,  $h_i$ and $u_i$, such that for
every $\delta \geq \delta_0$ we obtain 
$|s_i(u,u)| \leq c d_i(u,u)$ and therefore, the following
lemma is valid.

\begin{lemma} \label{lemmac}
There exists $\delta_0 > 0$ such that for $\delta \geq \delta_0$ and for all $u \in X(\Omega)$ we have, 
in each subdomain,
\begin{equation}
\gamma_0 d_i(u, u) \leq {a}^\prime_{i}(u,u) \leq \gamma_1 d_i(u, u), \quad 1 \leq i\leq N,
\end{equation}
and also we have the following global bilinear forms inequality
\begin{equation}
\gamma_0 d_h(u, u) \leq a_h(u, u) \leq \gamma_1 d_h(u, u).
\end{equation}
Here, $\gamma_0$ and $\gamma_1$ are positive constants independent of the 
$\rho_i, h_i$, $H_i$ and $u$. \\
\end{lemma}
The proof of this lemma is a modification of the proof of  \cite[Lemma 2.1]{MR2372024}, 
\cite{MR2914789} 
or \cite[Theorem 4.1]{MR2002258}, therefore it is omitted.

Lemma \ref{lemmac} implies that the discrete 
problem (\ref{discproblem}) is elliptic and continuous, therefore,
the solution exists and it is unique and stable. An optimal 
a priori error estimate of 
this method was established  in \cite{MR664882,MR1885715} for the continuous
coefficient case.
We mention here that Lemma \ref{lemmac} together with Lemma \ref{lemmaequiv}, see below,  
are going to be fundamental for establishing condition number estimates
for the FETI-DP preconditioned system developed in the remaining sections.

\section{Schur complement matrices and discrete harmonic extensions}
The first step of many iterative substructuring solvers, such as 
the FETI-DP method that we consider in this paper, requires
the elimination of unknowns interior to the subdomains. In this section,
we describe this step for DG discretizations. \\


We introduce some notation and then formulate (\ref{discproblem})
as a variational problem with constraints. Associated to a  
subdomain $\Omega_i$, we define the extended subdomain
${\Omega}_i^\prime$ by
\[
{\Omega}_i^\prime := \overline{\Omega}_i \;\bigcup\; 
\{\cup_{j \in {\cal{J}}_H^{i,0}} \bar{E}_{ji}\}
\]
i.e., it is the union of $\overline{\Omega}_i$ and the 
$\bar{E}_{ji} \subset \partial \Omega_j$ such that 
$j \in {\cal{J}}_H^{i,0}$, and the local interfaces ${\Gamma}_i$ by
\[
{\Gamma}_i: = \overline{\partial \Omega_i \backslash \partial \Omega}
~~~\mbox{and}~~~{\Gamma}_i^\prime : = \Gamma_i \;\bigcup \;
\{\cup_{j\in {\cal{J}}_H^{i,0}} \bar{E}_{ji} \}.
\]
We also introduce the sets
\begin{equation} \label{Gammadef} 
\Gamma := \bigcup_{i=1}^N {\Gamma}_i, ~~ 
\Gamma^\prime:=  \prod_{i=1}^N  {\Gamma}_i^\prime, 
~~ I_i :={\Omega}^\prime_i \backslash {\Gamma}_i^\prime
~~~\mbox{and}~~~ I: =  \prod_{i=1}^N I_i.
\end{equation}

\begin{figure}[htb]
\centering
{\psfig{figure=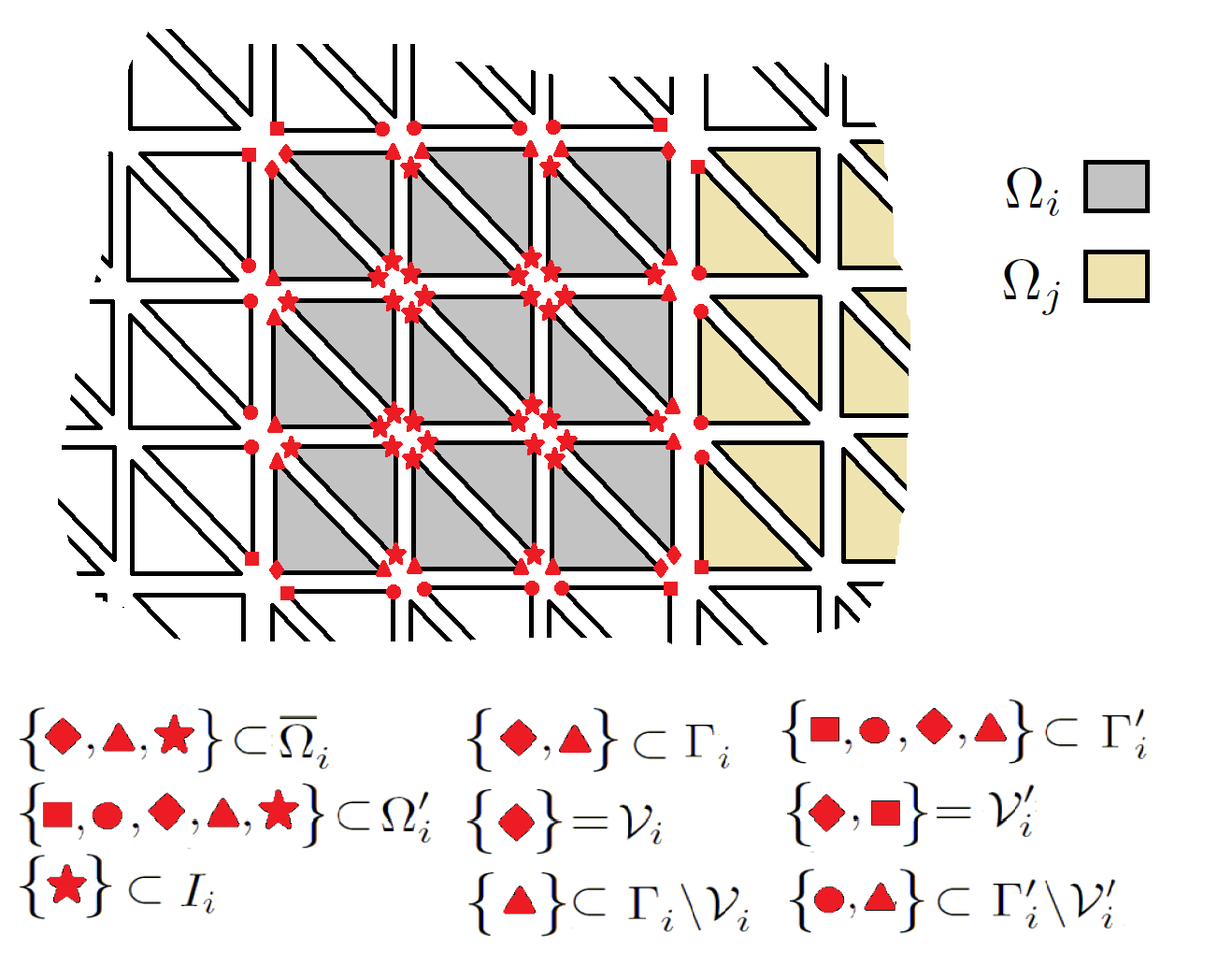,height=9cm,width=12cm,angle=0}}
\caption{ Illustration of the classification of degrees of freedom 
related to subdomain $\Omega_i$.}
\label{fig:dof}
\end{figure}

Associated to these sets, we classify degrees of freedom (DG nodal
values) on $\Omega_i^\prime$. We illustrate this classification 
(along with further classifications to be introduced later) in Figure
\ref{fig:dof}:
\begin{itemize}
\item $\Gamma_i$-degres of freedom: The nodal points which are endpoints of edges of
 ${\cal{E}}_h^{i,j}$ for $j \in {\cal{J}}_H^{i,0}$. 
\item Degrees ${\cal{V}}_i$:  The nodal points which are both 
an endpoint of an edges of  ${\cal{E}}_h^{i,j}$ for $j \in {\cal{J}}_H^{i,0}$
and a corner of $\Omega_i$. 
\item $\Gamma_i^\prime$-degrees of freedom: It is the union of the degrees
  $\Gamma_i$ and the nodal points which are endpoints of edges of
${\cal{E}}_h^{j,i}$ for  $j \in {\cal{J}}_H^{i,0}$.
\item ${\cal{V}}_i^\prime$-degrees of freedom:  It is the union of the degrees
  ${\cal{V}}_i$, and  the nodal points which are both an endpoint of an
  edge of ${\cal{E}}_h^{j,i}$ for  $j \in {\cal{J}}_H^{i,0}$ and a corner 
of $\Omega_j$. 
\item $I_i$-degrees of freedom: The nodal points which are vertices of 
elements $\tau$ of ${\cal{T}}_h^i$ and are not  endpoints of an
  edge of $\tau$ in ${\cal{E}}_h^{i,j}$ for  $j \in {\cal{J}}_H^{i,0}$. 
\item $\overline{\Omega}_i$-degrees of freedom: The union of
  $\Gamma_i$ and $I_i$ degrees of freedom. 
\item $\Omega_i^\prime$-degrees of freedom:  The union of
  $\Gamma_i^\prime$ and $I_i$ degrees of freedom.
\end{itemize} 

\begin{remark} Nodal points of elements in ${\cal{T}}_h^i$ which 
intersect $\bar{E}_{ij}$ 
by only one vertex are treated as $I_i$ degrees of freedom. The trace
of the basis DG functions associated to these 
nodal points are zero almost everywhere on $\Gamma_i$, hence, the 
analysis shows that it is convenient to treat these nodes as 
as $I_i$-degrees of freedom as defined above. For the same reason, nodal points of 
elements of ${\cal{T}}_h^j$ which intersect $\bar{E}_{ji}$ by vertex only for $j \in
  {\cal{J}}_H^{i,0}$, are not included as $\Omega_i^\prime$-degrees of
  freedom. We point out, however, that we could have considered 
these nodes points as $\Gamma_i$ and $\Omega_i^\prime$-degrees of
freedom, respectively as well, and by introducing also Lagrange multipliers to 
force continuity at these type of nodal values, and then 
eliminating these Lagrange multipliers and also these degrees
of freedom. This approach would be equivalent to the approach we 
consider and analyze here in this paper.   
\end{remark}

We mention that if we refer to the classical FETI-DP design, i.e., for continuous FEM, 
the set $I$ corresponds to the set of interior degrees of freedom (that are block-uncoupled) 
and can be eliminated in a first step. Moreover,  the
$\Gamma$ corresponds to the global interface with the original coupling and 
$\Gamma^\prime$ corresponds to the (block-wise) torn interface. We recall that 
the classical FETI-DP design extends the original problem to a problem in the torn 
interface space and then construct  and implement suitable restrictions to interconnect 
 back some of the broken coupling between degrees of freedom. We now describe these 
steps for DG discretization in detail. First we introduce the corresponding functions spaces.

 Let $W_i({\Omega}_i^\prime)$ be the  
FE space of functions defined by  
values on  ${\Omega}^\prime_i$
\begin{equation} \label{defiwi}
W_i({\Omega}^\prime_i) = W_i(\overline{\Omega}_i) \times
\prod_{j \in {\cal{J}}_H^{i,0}} W_i(\bar{E}_{ji}),
\end{equation}
where $W_i(\overline{\Omega}_i) := X_i(\Omega_i)$ and 
$W_i(\bar{E}_{ji})$ is the trace of the DG space $X_j(\Omega_j)$
on $\bar{E}_{ji}  \subset \partial 
\Omega_j$ for all $j \in {\cal{J}}_H^{i,0}$. A function
$ {u}^\prime_i \in W_i({\Omega}^\prime_i)$ is defined by the 
${\Omega}^\prime_i$-degrees of freedom. Below,
we  denote ${u}^\prime_i$ by $u_i$ if it is not confused with functions
of $X_i(\Omega_i)$. A function $ {u}_i \in W_i({\Omega}^\prime_i)$ 
is represented as 
\[
{u}_i = \{({u}_i)_i, \{({u}_i)_j\}_{j \in {\cal{J}}_H^{i,0}} \},
\]
where $(u_i)_i := {u}_{i\, | \overline{\Omega}_i}$ (${u}_i$ restricted
to $\overline{\Omega}_i$) and 
$(u_i)_j := {u}_{i\, | \bar{E}_{ji}}$ (${u}_i$ restricted
to  $\bar{E}_{ji}$).  
Here and below we use the same notation to identify 
both DG functions and their vector representations.  
Note that ${a}^\prime_i(\cdot,\cdot)$, see (\ref{aprimei}), is defined on 
$W_i({\Omega}^\prime_i) \times W_i({\Omega}^\prime_i)$ with corresponding 
stiffness matrix  ${A}^\prime_i$ defined by 
\begin{equation} \label{hataim}
{a}^\prime_i({u}_i, {v}_i) = \langle {A}^\prime_i {u}_i, 
{v}_i \rangle \quad {u}_i, {v}_i  \in W_i({\Omega}^\prime_i),
\end{equation}
where  $\langle {u}_i, {v}_i \rangle$ denotes the  $\ell_2$ 
inner product  of nodal values associated to the vector space in consideration.
We also represent 
${u}_i \in W_i({\Omega}^\prime_i)$ as ${u}_i =({u}_{i,I}, {u}_{i,\Gamma^\prime})$ where  
${u}_{i,\Gamma^\prime}$ represents values of ${u}_i$ at nodal 
points on ${\Gamma}_i^\prime$ and 
${u}_{i,I}$ at the interior nodal points in $I_i$, see (\ref{Gammadef}). 
 Hence, 
let us represent  $W_i({\Omega}^\prime_i)$ as the vector spaces
$W_i(I_i) \times W_i({\Gamma}_i^\prime)$.  Using the 
representation ${u}_i= ({u}_{i,I}, {u}_{i,\Gamma^\prime})$, the matrix 
${A}^\prime_i$ can be represented as 
     \begin{eqnarray} \label{Alocal}
{A}^\prime_{i} =
\left(
\begin{array}{cc}
{A}^\prime_{i,II} & {A}^\prime_{i,I\Gamma^\prime} \\
{A}^\prime_{i,\Gamma^\prime I} & {A}^\prime_{i,\Gamma^\prime\Gamma^\prime}
\end{array}
\right),
\end{eqnarray}
where the block rows and columns correspond to the 
 nodal points of $I_i$ and ${\Gamma}_i^\prime$, respectively. \\


The Schur complement of ${A}^\prime_i$ 
 with respect to $u_{i,\Gamma^\prime}$ is of the form
\begin{equation} \label{Si}
{S}^\prime_i  := {A}^\prime_{i,\Gamma^\prime\Gamma^\prime} - {A}^\prime_{i,\Gamma^\prime I}({A}^\prime_{i,II})^{-1} {A}^\prime_{i,I\Gamma^\prime}
\end{equation}
and introduce the block diagonal matrix $S^\prime=    \mbox{diag}\{S_i^\prime\}_{i=1}^N$.
Note that ${S}_i^\prime$ satisfies 
\begin{equation} \label{minloc}
\langle {S}^\prime_{i} {u}_{i,\Gamma^\prime}, {u}_{i,\Gamma^\prime} \rangle
\;=\; \min \; {a}^\prime_i({w}_i, {w}_i), 
\end{equation}
where the minimum is taken over $w_i =(w_{i,I},w_{i,\Gamma^\prime}) \in W_i({\Omega}^\prime_i)$ such that 
$w_{i,\Gamma^\prime} = u_{i,\Gamma^\prime}$ on ${\Gamma}^\prime_i$.
The bilinear form ${a}^\prime_i(\cdot,\cdot)$ is symmetric and nonnegative, 
see Lemma \ref{lemmac}. The minimizing function satisfying (\ref{minloc}) is 
called discrete harmonic in the sense of ${a}^\prime_i(\cdot,\cdot)$ or in 
the sense of ${\cal{H}}^\prime_i$. An equivalent definition of the 
minimizing function 
${\cal{H}}^\prime_i u_{i,\Gamma^\prime} \in  W_i({\Omega}^\prime_i)$ is given by 
the solution of
\begin{equation} \label{hatHi}
{a}^\prime_i({\cal{H}}^\prime_i u_{i,\Gamma^\prime}, v_{i,\Gamma^\prime}) =
 0 \qquad v_{i,\Gamma^\prime} \in 
{\stackrel{o}{W}}_i({\Omega}^\prime_i),
\end{equation}
\begin{equation} \label{hatHibdry}
{\cal{H}}^\prime_i u_{i,\Gamma^\prime} = u_{i,\Gamma^\prime}  \quad \mbox{on} \;\; {\Gamma}^\prime_i,
\end{equation}
where ${\stackrel{o}{W}}_i({\Omega}^\prime_i)$ is the subspace of 
${W}_i({\Omega}^\prime_i)$ of functions which vanish on ${\Gamma}^\prime_i$.
We note that for substructures $\Omega_i$ which
intersect $\partial \Omega$ by edges, the nodal values
of ${W}_i({\Omega}^\prime_i)$ on 
$\partial \Omega_i \backslash {\Gamma}_i\subset \partial\Omega$ are 
treated as  unknowns and belong to $I_i$. \\

  Let us introduce the product space 
\begin{equation}
W(\Omega^\prime) := \prod^N_{i=1} W_i({\Omega}^\prime_i),
\end{equation}
i.e., $u \in W(\Omega^\prime)$ means that 
$u = \{{u}_i\}_{i=1}^N$ where $u_i \in W_i({\Omega}^\prime_i)$; see
(\ref{defiwi}) for the definition of  $W_i({\Omega}_i^\prime)$. 
 Recall  that we write $(u_i)_i = {u}_{i\, | \overline{\Omega}_i}$ (${u}_i$ restricted
to $\overline{\Omega}_i$) and 
$(u_i)_j= {u}_{i\, | \bar{E}_{ji}}$ (${u}_i$ restricted
to  $\bar{E}_{ji}$). Using the representation $u_i = (u_{i,I}, u_{i,\Gamma^\prime})$ where
$u_{i,I} \in W_i(I_i)$ and $u_{i,\Gamma^\prime} \in W_i(\Gamma_i^\prime)$, see (\ref{Alocal}), 
let us introduce the product space
\begin{equation} 
W(\Gamma^\prime) :=\prod^N_{i=1} W_i({\Gamma}_i^\prime),
\end{equation}
i.e., $u_{\Gamma^\prime} \in  W(\Gamma^\prime)$ means that 
$u_{\Gamma^\prime} = \{u_{i,\Gamma^\prime}\}_{i=1}^N$ where $u_{i,\Gamma^\prime}
\in W_i({\Gamma}^\prime_i)$. The space $W(\Gamma^\prime)$ which
was defined on $\Gamma^\prime$ only, is also  
interpreted below as the subspace of ${W}({\Omega}^\prime)$ of functions 
which are  
discrete harmonic in the sense of 
${\cal{H}}_i^\prime$ in each $\Omega_i$. \\

We now consider the subspaces
$\hat{W}({\Omega}^\prime) \subset W({\Omega}^\prime)$ and $\hat{W}({\Gamma}^\prime) \subset W({\Gamma}^\prime)$
 as the space of functions which are  continuous
 on $\Gamma$ in the sense of the following definition (for notation see 
(\ref{Gammadef})).  \\

\begin{figure}[htb]
\centering
{\psfig{figure=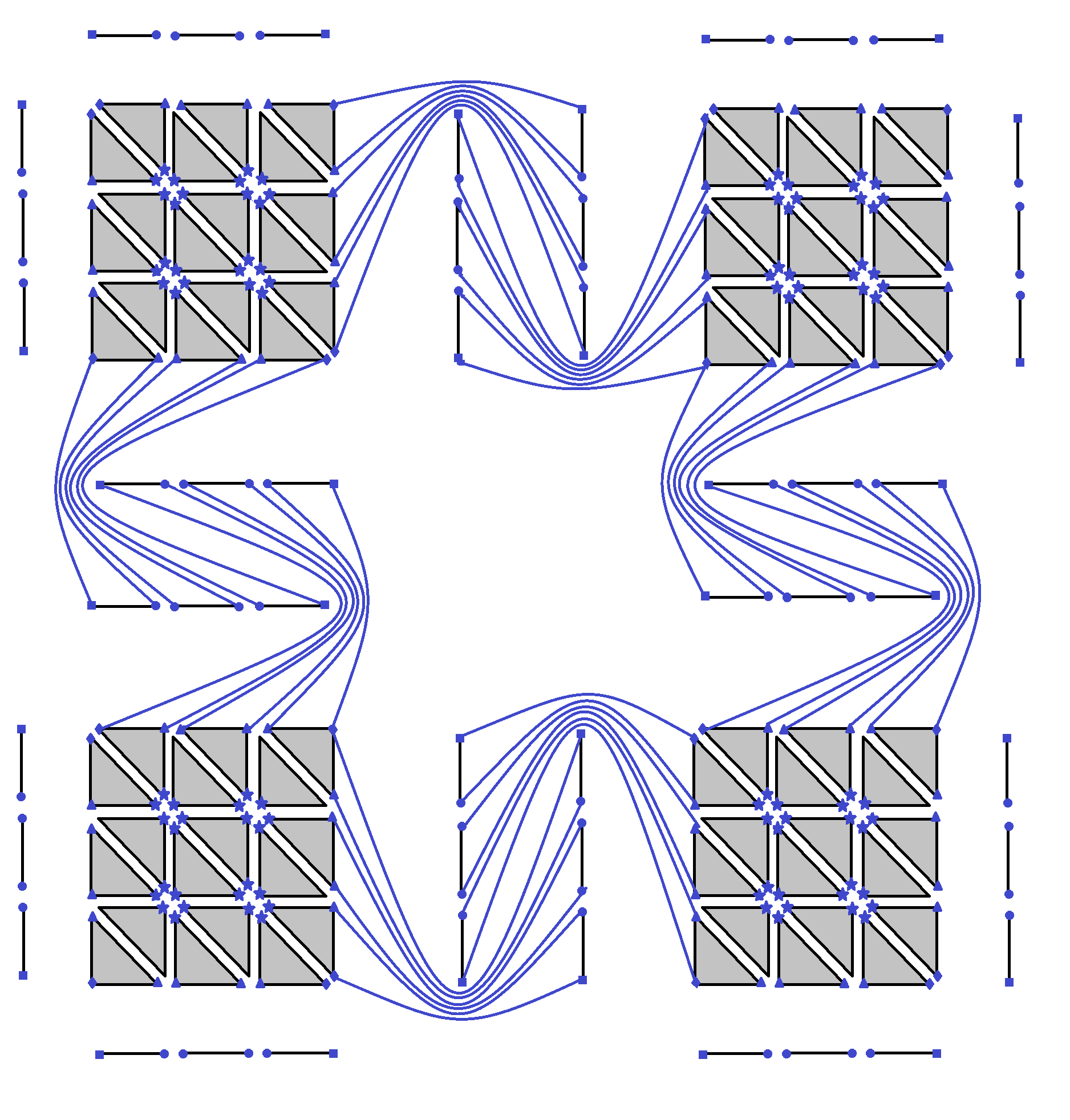,height=8cm,width=10cm,angle=0}}
\caption{ Illustration of the continuity on $\Gamma$.}
\label{fig:cont}
\end{figure}

\begin{definition}[Spaces $\hat{W}({\Omega^\prime})$ and  $\hat{W}({\Gamma}^\prime)$] \label{defini1} 
We say that $u = \{u_i\}_{i=1}^N \in W({\Omega}^\prime)$ 
is  continuous  on $\Gamma$ if for all $i$, $1\leq i\leq N$, we have  
\begin{equation} \label{continGamma}
(u_i)_i(x) = (u_j)_i(x)\;\;\;\; \mbox{for all} \;\;x \in \bar{E}_{ij} \;\; \mbox{for all} \;\; 
j \in  \,{\mathcal{J}}_H^{i,0}
\end{equation}
and
\begin{equation} \label{contoutGamma}
(u_i)_j(x) = (u_j)_j(x) \;\;\;\;  \mbox{for all} \;\;x \in \bar{E}_{ji} \;\; \mbox{for all} \;\; 
j \in  \, {\mathcal{J}}_{H}^{i,0}.
\end{equation} 
In Figure \ref{fig:cont} we illustrate this continuity by assigning 
the same nodal value for nodes connected by a line. The 
subspace of ${W}({\Omega^\prime})$  
of continuous functions on $\Gamma$ is denoted by $\hat{W}({\Omega}^\prime)$, 
and the subspace of $\hat{W}({\Omega}^\prime)$ of 
functions which are  
discrete harmonic in the sense of 
${\cal{H}}_i^\prime$ in each $\Omega_i$ is 
denoted by  $\hat{W}({\Gamma}^\prime)$. \\
 \end{definition}

Note that there is a one-to-one correspondence between vectors 
in the spaces $X(\Omega)$ and $\hat{W}(\Omega^\prime)$. Indeed, let
us introduce the restriction matrices
 ${R}_{{\Omega}^\prime_i}: X(\Omega) 
\rightarrow W_i({\Omega}^\prime_i)$ which 
assign uniquely the vector values of $u = \{u_i\}_{i=1}^N \in X(\Omega)$ 
where $u_i \in X_i(\Omega_i)$, see (\ref{defXi}), to
$v_i \in W_i({\Omega}^\prime_i)$ defined by 
$({v}_i)_i = u_i$ on $\overline{\Omega}_i$
and $({v}_i)_j = u_j$ on $\bar{E}_{ji}$ for all $j \in {\mathcal{J}}_H^{i,0}$.
It is easy to see that $v := \{{v}_i\}_{i=1}^N$ where
$v_i := {R}_{{\Omega}^\prime_i} u$ belongs to $\hat{W}(\hat{\Omega})$.
And vice-versa, for each 
 $v = \{{v}_i\}_{i=1}^N \in \hat{W}(\Omega)$, we can define uniquely 
$u = \{u_i\}_{i=1}^N \in X(\Omega)$ by setting $u_i = (v_i)_i$. Since
$u$ and $v$ have identical nodal values, we  refer sometimes
$u \in \hat{W}(\Omega^\prime)$ or $u \in X(\Omega)$. For instance,
the solution $u^*$ of  (\ref{discproblem}) can be interpreted 
as a function in $\hat{W}(\Omega^\prime)$ or in $X(\Omega)$.  \\

  Note that the discrete problem (\ref{discproblem}) can
be written as a system of algebraic equations
\begin{equation} \label{algebraic}
\hat{A}u^* = f
\end{equation}
for $u^* \in  X(\Omega)$ using the standard FE basis functions, 
and $f =\{f_i\}_{i=1}^N \in  X(\Omega)$,
where $f_i$ is the load vector associated with individual subdomains
$\Omega_i$, i.e., is $\int_{\Omega} f v_i$ when $v_i$ are 
the canonical basis functions of $X_i(\Omega_i)$.  The 
stiffness matrix $\hat{A}$ can be obtained
by assembling the matrices ${A}_i^\prime$,  see (\ref{hataim}), from 
$W({\Omega}^\prime)$ to $X(\Omega)$ as
\[
\hat{A} = \sum_{i=1}^N {R}_{{\Omega}^\prime_i}^T {A}^\prime_i {R}_{{\Omega}^\prime_i}.
\]
Note that the matrix 
$\hat{A}$ is not block diagonal since there are couplings 
between substructures due to the continuity on $\Gamma$, see Defintion 
\ref{defini1}. \\

Note also that $X(\Omega)$ can be  componentwise represented  by
$X(I) \times X(\Gamma)$, denoted also by $X(\Omega)$, where 
  $X(I) := \prod_{i=1}^N X_i(I_i)$  is the vector space of
functions defined by nodal values on $I_i$, and 
$X(\Gamma) := \prod_{i=1}^N X_i(\Gamma_i)$ by the nodal
values on $\Gamma_i$, see (\ref{Gammadef}) and Figure \ref{fig:dof}.
 Hence, we can represent
$u \in X(\Omega)$  
as $u = (u_I, u_\Gamma)$ with $u_I \in X(I)$ and $u_\Gamma \in X(\Gamma)$.  
We introduce the 
restriction $R_{{\Gamma}^\prime_i}: X(\Gamma) \rightarrow {W}_i({\Gamma}^\prime_i)$
by assigning values $u_\Gamma \in  X(\Gamma)$ into $u_i \in 
{W}_i({\Gamma}^\prime_i)$ at the nodes of $\Gamma_i^\prime$. 
 By eliminating the variable $u_I^* = \{u^*_{i,I}\}_{i=1}^N$ 
of $u^* = (u^*_I,  u^*_\Gamma)$  
from (\ref{algebraic}), see (\ref{Alocal}) and (\ref{Si}), it is easy to see
that 
\begin{equation} \label{SchurS}
\hat{S}u^*_\Gamma = \hat{g}_\Gamma
\end{equation}
where
\begin{equation} \label{assemblyingS} 
\hat{S} = \sum_{i=1}^N 
 {{R}}_{{\Gamma}^\prime_i}^T  {S}^\prime_i {{R}}_{{\Gamma}^\prime_i}\quad \mbox{and} \quad  \hat{g}_\Gamma = {f}_\Gamma - \sum_{i=1}^N 
  {{R}}_{{\Gamma}^\prime_i}^T {A}^\prime_{i, \Gamma^\prime I} 
({A}^\prime_{i, II})^{-1} f_i,
\end{equation}
 with ${f}_\Gamma := \{f_{i,\Gamma}\}_{i=1}^N$ and $f_{i,\Gamma} \in
X_i(\Gamma_i)$. Here, the load vector $f_{i,\Gamma}$ is defined by 
$\int_{\Omega_i} f v_{i,\Gamma}$ when $v_{i,\Gamma}$ are the cannonical basis
functions of $X_i(\Omega_i)$ associated to nodes on $\Gamma_i$. It is also 
easy to see from  (\ref{hatHi}) and (\ref{hatHibdry}) that 
\begin{eqnarray} \label{equivSandA}
\left(
\begin{array}{c}
v_{i,I} \\
v_{i,\Gamma^\prime} 
\end{array}
\right)^T 
\left(
\begin{array}{cc}
{A}^\prime_{i,II} & {A}^\prime_{i,I\Gamma^\prime} \\
{A}^\prime_{i,\Gamma^\prime I} & {A}^\prime_{i,\Gamma^\prime\Gamma^\prime}
\end{array}
\right)  
\left(
\begin{array}{c}
{\mathcal{H}}_i^\prime  u_{i,\Gamma^\prime} \\
u_{i,\Gamma^\prime} 
\end{array}
\right)  
=\;
\langle {S}^\prime_i  u_{i,\Gamma^\prime} ,v_{i,\Gamma^\prime} \rangle.
\end{eqnarray}

Note that $\hat{W}(\Gamma^\prime)$ is the natural space for 
defining $\langle\hat{S}\; \cdot, \cdot \rangle$ due to (\ref{assemblyingS}), (\ref{equivSandA}) and 
the  continuity   of $\hat{W}(\Gamma^\prime)$ on $\Gamma$, see 
Definition \ref{defini1}.

\section{FETI-DP with corner constraints}
We now design a FETI-DP method for solving (\ref{SchurS}). We 
follow the abstract approach described in pages 160-167 in \cite{MR2104179}. \\

 For $1\leq i \leq N$, we introduce the nodal points associated to the corner
unknowns, see Figure \ref{fig:dof}, by  
\begin{equation} \label{tildenui}
{\mathcal{V}}^\prime_i: =\mathcal{V}_i \;\bigcup \;\{\cup_{j \in 
\mathcal{J}_H^{i,0}} \;\partial E_{ji}\} 
\qquad \mbox{where} \qquad 
{\mathcal{V}}_i := \{ \cup_{j \in {\mathcal{J}}_H^{i,0}}  \;\partial E_{ij}\}.
\end{equation}

We now consider the subspace 
$\tilde{W}({\Omega}^\prime) \subset W({\Omega}^\prime)$ and $\tilde{W}({\Gamma}^\prime) \subset W({\Gamma}^\prime)$
 as the space of functions which are  continuous
 on all the ${\mathcal{V}}_i$ in the sense of the following 
definition:  \\

\begin{figure}[htb]
\centering
{\psfig{figure=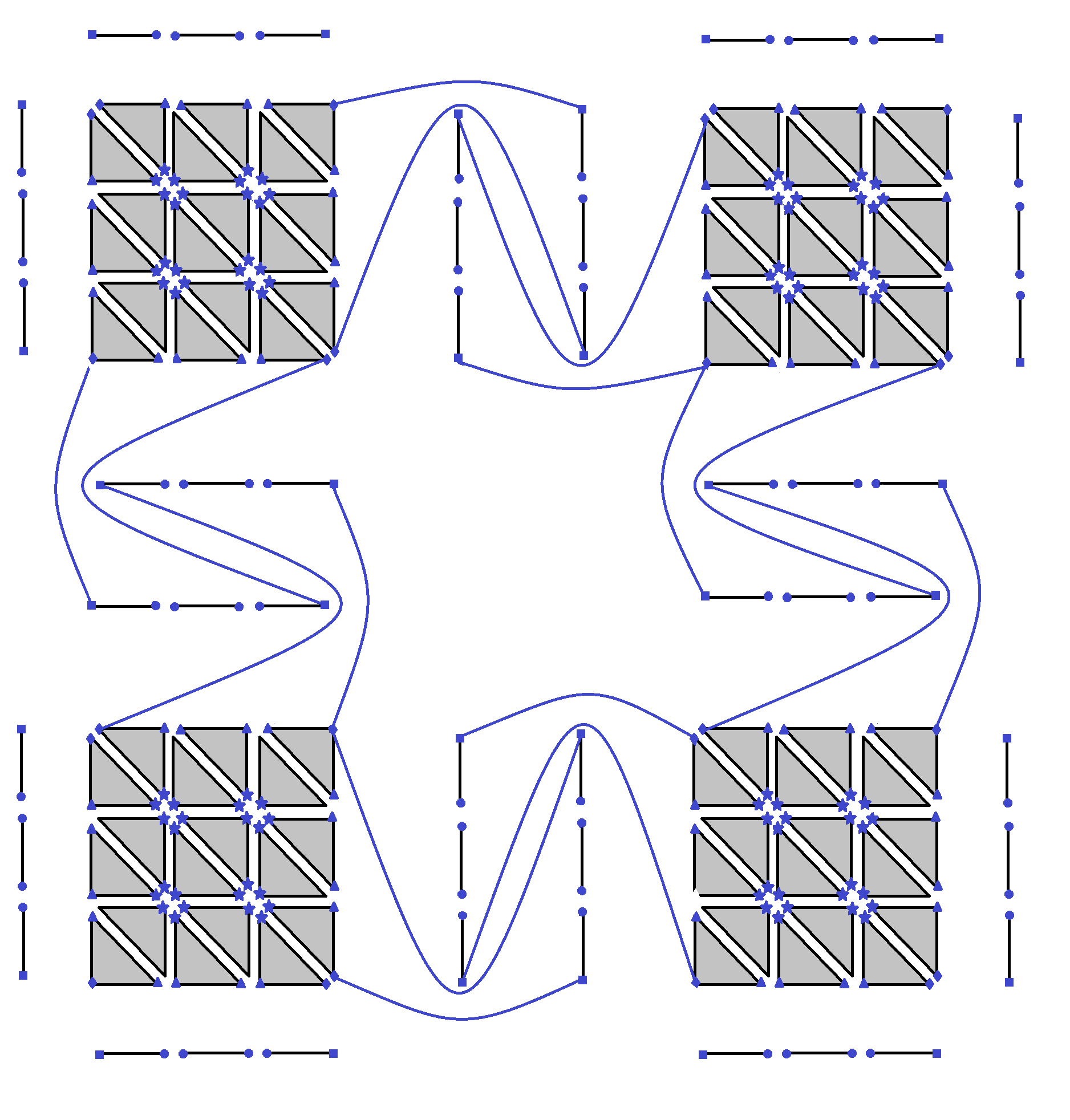,height=8cm,width=10cm,angle=0}}
\caption{Illustration of the continuity at the corners.}
\label{fig:contV}
\end{figure}

\begin{definition}[Subspaces $\tilde{W}({\Omega^\prime})$ and 
$\tilde{W}({\Gamma^\prime})$] \label{defini2} 
We say that $u = \{u_i\}_{i=1}^N \in W({\Omega^\prime})$ 
is  continuous  at the corners ${\mathcal{V}}_i^\prime$ if 
for $1 \leq i \leq N$ we have  
\begin{equation} \label{contin}
(u_i)_i(x) = (u_j)_i(x)\;\;\;\; \mbox{at} \;\;x \in \partial E_{ij} \;\; \mbox{for all} \;\; 
j \in \,{\mathcal{J}}_H^{i,0}
\end{equation}
and
\begin{equation} \label{contout}
(u_i)_j(x) = (u_j)_j(x) \;\;\;\;  \mbox{at} \;\;x \in \partial E_{ji} \;\; \mbox{for all} \;\; 
j \in  \, {\mathcal{J}}_H^{i,0}.
\end{equation}
In Figure \ref{fig:contV} we illustrate this continuity by 
assigning the same 
nodal value at nodes (corners) connected by a line.
The subspace of ${W}({\Omega^\prime})$ of continuous functions 
at the corners 
 ${\mathcal{V}}^\prime_i$ for all $1\leq i\leq N$ 
is denoted by  $\tilde{W}({\Omega^\prime})$,
and the subspace of  $\tilde{W}({\Omega^\prime})$ of functions which are 
discrete harmonic in the sense of 
${\cal{H}}_i^\prime$ is denoted by  $\tilde{W}({\Gamma^\prime})$. \\
 \end{definition} 

Note that 
\begin{equation} \label{inclusion}
\hat{W}(\Gamma^\prime) \subset \tilde{W}(\Gamma^\prime) \subset W(\Gamma^\prime).
\end{equation}

Let $\tilde{A}$  be the stiffness matrix which is obtained by assembling 
the matrices ${A}_i^\prime$ for $1 \leq i \leq N$, from $W(\Omega^\prime)$ to 
$\tilde{W}(\Omega^\prime)$. Note that the matrix  
$\tilde{A}$ is no longer block diagonal since there are couplings 
between variables at the corners  ${\mathcal{V}}_i^\prime$ for $1\leq i\leq N$. 
We will represent 
$u \in \tilde{W}(\Omega^\prime)$ as $u = (u_I, u_\Pi, u_\triangle)$ where the 
subscript
$I$ refers to the interior degrees of freedom at nodal points $I = \prod_{i=1}^N I_i$, 
the $\Pi$ refers to the corners ${\mathcal{V}}_i^\prime$ for all $1\leq i\leq N$,
and the $\triangle$ refers 
to the remaining nodal points, i.e., the nodal points 
${\Gamma}_i^\prime  \backslash {\mathcal{V}}_i^\prime$, for all 
$1\leq i\leq N$.  The vector $u = (u_I, u_\Pi, u_\triangle) \in \tilde{W}(\Omega^\prime)$ is obtained from the vector $u = \{u_i\}_{i=1}^N 
\in W(\Omega^\prime)$ using the equations (\ref{contin}), (\ref{contout}), i.e., 
the continuity of $u$ on $ {\mathcal{V}}_i^\prime$ for all $1\leq i \leq N$.
Using 
the decomposition $u = (u_I, u_\Pi, u_\triangle) \in \tilde{W}(\Omega^\prime)$ 
we can partition $\tilde{A}$ as
\begin{eqnarray} \label{tildeA}
\tilde{A} =
\left(
\begin{array}{ccc}
{A}^\prime_{II} & {A}^\prime_{I\Pi} & {A}^\prime_{I\triangle} \\
{A}^\prime_{\Pi I}
 & \tilde{A}_{\Pi\Pi} & {A}^\prime_{\Pi \triangle} \\
{A}^\prime_{\triangle I } & {A}^\prime_{\triangle \Pi} & {A}^\prime_{\triangle \triangle}
\end{array}
\right). 
\end{eqnarray}
We note that the only couplings across subdomains are through the
variables $\Pi$ where the matrix $\tilde{A}$ is subassembled.   \\

A Schur complement of $\tilde{A}$ with respect to the $\triangle$-unknowns (eliminating the $I$- and the $\Pi$-unknowns) is of the form
\begin{eqnarray} \label{tildeS} 
\tilde{S} := {A}^\prime_{\triangle\triangle} - ({A}^\prime_{\triangle I} ~ {A}^\prime_{\triangle \Pi})
\left(
\begin{array}{cc}
{A}^\prime_{II} & {A}^\prime_{I\Pi} \\
{A}^\prime_{\Pi I} & \tilde{A}_{\Pi\Pi}
\end{array}
\right)^{-1}
\left(
\begin{array}{c}
{A}_{I\triangle}^\prime \\ {A}_{\Pi\triangle}^\prime
\end{array}
\right).
\end{eqnarray}

  A vector $u \in \tilde{W}(\Gamma^\prime)$ can uniquely be represented by 
$u = (u_\Pi,u_\triangle)$, therefore, we can represent 
$\tilde{W}(\Gamma^\prime)= \hat{W}_\Pi(\Gamma^\prime) \times 
 {W}_\triangle(\Gamma^\prime)$, 
where
$\hat{W}_\Pi(\Gamma^\prime)$ refers to the $\Pi$-degrees of freedom of 
$\tilde{W}(\Gamma^\prime)$ while ${W}_\triangle(\Gamma^\prime)$ to the 
$\triangle$-degrees of freedom of $\tilde{W}(\Gamma^\prime)$. The vector space 
${W}_\triangle (\Gamma^\prime)$ can be decomposed as 
\begin{equation} \label{productdelta}
{W}_\triangle (\Gamma^\prime) = \prod^N_{i=1} {W}_{i,\triangle}({\Gamma}_i^\prime)
\end{equation}
where the local space ${W}_{i,\triangle}({\Gamma}_i^\prime)$ refers to the degrees of freedom 
associated to the nodes of ${\Gamma}_i^\prime \backslash {\mathcal{V}}_i^\prime$ for
$i=1\leq i\leq N$.
Hence, 
a  vector $u \in \tilde{W}(\Gamma^\prime)$ can be represented as $u = (u_\Pi, u_\triangle)$ 
with $u_\Pi \in \hat{W}_\Pi(\Gamma^\prime)$ and 
$u_\triangle = \{u_{i,\triangle} \}^N_{i=1} \in {W}_\triangle(\Gamma^\prime)$ 
where $u_{i,\triangle} \in 
{W}_{i,\triangle}({\Gamma}_i^\prime)$. Note that $\tilde{S}$, see (\ref{tildeS}), is 
defined on the vector space ${W}_\triangle(\Gamma^\prime)$, and 
the following lemma
follows (cf. Lemma 6.22 in \cite{MR2104179} and Lemma 4.2 in \cite{MR1835470}): \\
\begin{lemma} \label{lemma41}
Let $u_\triangle \in {W}_\triangle(\Gamma^\prime)$ and let $\tilde{S}$ 
and $\tilde{A}$, defined in (\ref{tildeS}) and (\ref{tildeA}), respectively.
Then, 
\begin{equation}
\langle\tilde{S} u_\triangle, u_\triangle \rangle\; = \min \langle \tilde{A}w, 
w\rangle 
\end{equation}
where the minimum is taken over  $w = (w_I, w_\Pi, w_\triangle) \in \tilde{W}(\Omega)$ such that $w_\triangle = u_\triangle$. \\
\end{lemma}

   Let us take $u \in \tilde{W}(\Gamma^\prime)$ as 
$u = (u_\Pi, u_\triangle)$ with 
$u_\Pi \in \hat{W}_\Pi(\Gamma^\prime)$ and $u_\triangle \in {W}_\triangle(\Gamma^\prime)$, where $u_\triangle = \{u_{i,\triangle}\}^N_{i=1}$ with  
$u_{i,\triangle} \in {W}_{i,\triangle}({\Gamma}_i^\prime)$. The vector $u_{i,\triangle} \in 
{W}_{i,\triangle}({\Gamma}_i^\prime)$ can be partitioned as
$$
u_{i,\triangle} = \{(u_{i,\triangle})_i, \{(u_{i,\triangle})_j\}_{j \in 
{\mathcal{J}}_{H}^{i,0}}\}
$$
where
$$
(u_{i,\triangle})_i = u_{i,\triangle|_{\Gamma_i \backslash {\mathcal{V}}_i}}
 \;\; \mbox{and} \;\; (u_{i,\triangle})_j = u_{i,\triangle |_{{E}_{ji}}}.
$$

In order to measure the jump of $u_\triangle \in 
{W}_{\triangle}(\Gamma^\prime)$ across the $\triangle$-nodes  let 
us introduce the space $\hat{W}_{\triangle}(\Gamma)$ defined by  
\[
\hat{W}_{\triangle}(\Gamma) =  \prod_{i=1}^N 
X_i(\Gamma_i \backslash {\mathcal{V}}_i),
\]
where $X_i(\Gamma_i \backslash \mathcal{V}_i)$ is the restriction
of $X_i(\Omega_i)$ to $\Gamma_i \backslash {\mathcal{V}}_i$, see
Figure \ref{fig:dof}. To define the jumping matrix $B_\triangle: 
{W}_\triangle(\Gamma^\prime)
\rightarrow \hat{W}_\triangle(\Gamma)$,  
let $u_\triangle = \{ u_{i,\triangle}\}_{i=1}^N \in 
{W}_\triangle(\Gamma^\prime)$ and let $v := B_\triangle u$
where $ v = \{v_i\}_{i=1}^N 
\in \hat{W}_\triangle(\Gamma)$ is defined by 
\begin{equation} \label{jump1}
v_{i} = (u_{i,\triangle})_i - (u_{j,\triangle})_i  \;\; \mbox{on} \;\;
E_{ijh}\;\; \mbox{for all }\;\; j \in {\mathcal{J}}_H^{i,0}.
\end{equation}
The jumping matrix $B_\triangle$ can be written as 
\begin{equation} \label{Bdelta}
B_\triangle = (B^{(1)}_\triangle, 
 B^{(2)}_\triangle, \cdots, B^{(N)}_\triangle),
\end{equation}
where the rectangular matrix 
$B^{(i)}_\triangle$ consists of  columns 
of $B_\triangle$ attributed to the $(i)$ components of functions of  
${W}_{i,\triangle}(\Gamma^\prime_i)$ 
of the product 
space ${W}_\triangle(\Gamma^\prime)$, see (\ref{productdelta}). 
The 
 entries of 
the rectangular matrix consist of values of $\{0,1,-1\}$.
 It is easy to see that the 
$\mbox{Range}~ B_\triangle = \hat{W}_\triangle(\Gamma)$, so
$B_\triangle$ is full rank. In addition, if $u = (u_\Pi,u_\triangle) \in 
\tilde{W}(\Gamma^\prime)$ and $B_\triangle u_\triangle = 0$ then 
$u \in \hat{W}(\Gamma^\prime)$. \\

 We can reformulate the problem (\ref{SchurS}) as the variational problem with 
constraints  in ${W}_\triangle(\Gamma^\prime)$ space:  \emph{Find $u^*_\triangle \in {W}_\triangle(\Gamma^\prime)$ such that}
\begin{equation} \label{infJ}
J(u^*_\triangle) = \min J(v_\triangle)
\end{equation}
subject to $v_\triangle \in {W}_\triangle(\Gamma^\prime)$ with constraints $B_\triangle v_\triangle = 0$, where
\begin{equation}
J(v_\triangle) := 1/2 \langle \tilde{S} v_\triangle, v_\triangle\rangle - 
\langle \tilde{g}_\triangle, v_\triangle\rangle
\end{equation}
where $\tilde{S}$ is defined in (\ref{tildeS}) and 
\begin{eqnarray} \label{gtildetri}
\tilde{g}_\triangle := {f}_\triangle - ({A}^\prime_{\triangle I} ~ 
{A}^\prime_{\triangle \Pi})
\left(
\begin{array}{cc}
{A}^\prime_{II} & {A}^\prime_{I\Pi} \\
{A}^\prime_{\Pi I} & \tilde{A}_{\Pi\Pi}
\end{array}
\right)^{-1}
\left(
\begin{array}{c}
f_I \\ {f}_\Pi
\end{array}
\right).
\end{eqnarray}
We note that $f = \{f_i\}_{i=1}^N \in X(\Omega)$ was defined in
(\ref{algebraic}) and it can be represented
as $f = (f_I, f_\Pi, f_{\Gamma \backslash \Pi})$. It remains
to define  ${f}_\triangle$ in (\ref{gtildetri}). The forcing term
 ${f}_\triangle \in W_\triangle(\Gamma^\prime)$ is defined by 
${f}_\triangle = \{{f}_{i,\triangle}\}_{i=1}^N$ where 
${f}_{i,\triangle}$ are the load vectors associated with the 
individual subdomains $\Omega_i$, i.e., the entries ${f}_{i,\Delta}$
are defined as $\int_{\Omega_i} f v_{i,\Delta} dx$ when $v_{i,\Delta}$
are the canonical basis functions of ${W}_{i,\Delta}({\Gamma}_i^\prime)$. 
Note that 
$\tilde{S}$ is symmetric and positive definite
since $\tilde{A}$ has 
these properties; see also Lemma \ref{lemma41}. Introducing 
Lagrange multipliers 
$\lambda \in \hat{W}_\triangle(\Gamma)$, 
the problem (\ref{infJ}) reduces 
to the saddle point problem of the form: \emph{Find 
$u^*_\triangle \in {W}_\triangle(\Gamma^\prime)$ and $\lambda^* \in 
\hat{W}_\triangle(\Gamma)$
 such that}
\begin{eqnarray} \label{saddle}
\left\{
\begin{array}{ccccc}
\tilde{S} u^*_\triangle &+& B^T_\triangle \lambda^* & = &\tilde{g}_\triangle \\ 
B_\triangle u^*_\triangle & & &  = & 0.
\end{array}
\right.
\end{eqnarray}
Hence, (\ref{saddle}) reduces to 
\begin{equation} \label{ProjF0}
F\lambda^* = g
\end{equation}
where
\begin{equation} \label{ProjF}
F := B_\triangle \tilde{S}^{-1} B^T_\triangle, \qquad g:= B_\triangle \tilde{S}^{-1} \tilde{g}_\triangle.
\end{equation}
When $\lambda^*$ is computed, $u^*_\triangle$ can be found by solving the
 problem
\begin{equation} \label{Fandg}
\tilde{S} 
 u^*_\triangle = \tilde{g}_\triangle - B^T_\triangle \lambda^*.
\end{equation} \\

\subsection{Dirichlet Preconditioner}

We now define the FETI-DP preconditioner for $F$, see (\ref{ProjF}). 
Let ${S}^\prime_{i,\triangle}$ be the 
Schur complement of ${S}^\prime_i$, see (\ref{Si}), restricted to 
${W}_{i,\triangle}({\Gamma}_i^\prime) \subset {W}_i({\Gamma}_i^\prime)$, 
i.e., taken ${S}^\prime_i$ on functions in ${W}_i({\Gamma}_i^\prime)$ which 
vanish on $   {\mathcal{V}}_i^\prime$. Let
\begin{equation} \label{Stri}
{S}^\prime_\triangle :=     \mbox{diag} \{{S}^\prime_{i,\triangle}\}^N_{i=1}.
\end{equation}
In other words, ${S}^\prime_{i,\triangle}$ is obtained from ${S}^\prime_i$ by deleting rows and columns corresponding to nodal values at nodal points of 
$ {\mathcal{V}}_i^\prime \subset {\Gamma}_i^\prime$. \\

Let us introduce diagonal scaling matrices 
$D^{(i)}_\triangle: {W}_{i,\triangle}(\Gamma_i^\prime) \rightarrow
{W}_{i,\triangle}(\Gamma_i^\prime)$, for $1 \leq i \leq N$. 
The diagonal entry of $D^{(i)}_\triangle$ associated to a node $x \in 
\Gamma_i^\prime \backslash {\mathcal{V}}_i^\prime$,  
which we denote by $D^{(i)}_\triangle(x)$, is defined  by 
\begin{equation} \label{diagscal} 
D^{(i)}_{\triangle}(x) = \frac{\rho^{\beta}_j}
{\rho^{\beta}_i + \rho^{\beta}_j}\;\; \mbox{for}\;\;x \in \{E_{ijh} \cup E_{jih}\} 
\;\; \mbox{for}\;\;j \in {\cal{J}}_H^{i,0}, 
\end{equation}
for $\beta \in [1/2,\infty)$, see \cite{MR1469678}, and define 
\begin{equation} \label{BD}
B_{D,\triangle} = ( B^{(1)}_\triangle D^{(1)}_\triangle, \cdots , B^{(N)}_\triangle D^{(N)}_\triangle).
\end{equation}
Let 
\begin{equation} \label{Ptri}
P_\triangle := B^T_{D,\triangle} B_\triangle
\end{equation}
which maps ${W}_\triangle(\Gamma^\prime)$ into itself. It is easy to check 
that  for $w_\triangle = \{ w_{i,\triangle}\}^N_{i=1} \in {W}_\triangle(\Gamma^\prime)$ and $v_\triangle := P_\triangle w_\triangle$, the following
equalities hold:
\begin{equation} \label{Pdelta1}
(v_{i,\triangle})_i(x) = \frac{\rho^{\beta}_j}{\rho^{\beta}_i + \rho^{\beta}_j} [(w_{i,\triangle})_i(x) - (w_{j,\triangle})_i(x)], \qquad x \in E_{ijh},
\end{equation}
\begin{equation} \label{Pdelta2}
(v_{i,\triangle})_j(x) = \frac{\rho^{\beta}_j}{\rho^{\beta}_i + \rho^{\beta}_j} [(w_{i,\triangle})_j(x) - (w_{j,\triangle})_j(x)], \qquad x \in E_{jih}.
\end{equation}
Note that 
\begin{equation} \label{Pdelta3}
(v_{j,\triangle})_j(x) = \frac{\rho^{\beta}_i}{\rho^{\beta}_i + \rho^{\beta}_j} [(w_{j,\triangle})_j (x) - (w_{i,\triangle})_j (x)], \qquad x \in E_{jih},
\end{equation} 
\begin{equation}  \label{Pdelta4}
(v_{j,\triangle})_i(x) = \frac{\rho^{\beta}_i}{\rho^{\beta}_i + \rho^{\beta}_j} [(w_{j,\triangle})_i (x)) - (w_{i,\triangle})_i (x))], \qquad x \in E_{ijh}.
\end{equation}

By subtracting (\ref{Pdelta4}) from (\ref{Pdelta1}) and 
(\ref{Pdelta2}) from (\ref{Pdelta3}) we see 
that $P_\triangle$ preserves jumps in the sense that 
\begin{equation} \label{preserve} 
B_\triangle P_\triangle = B_\triangle. 
\end{equation}
From this follows that 
$P_\triangle$ is a projection ($P_\triangle^2 = P_\triangle$). \\

In the FETI-DP method,  the preconditioner $M^{-1}$ is defined
as follows: 
\begin{equation} \label{Minv}
M^{-1} = B_{D,\triangle} {S}^\prime_\triangle B^T_{D,\triangle} = \sum^N_{i=1} 
B^{(i)}_\triangle D^{(i)}_\triangle {S}^\prime_{i,\triangle} D^{(i)}_\triangle 
(B^{(i)}_\triangle)^T.
\end{equation}
Note that $M^{-1}$ is a block-diagonal matrix, and each block is 
invertible since ${S}_{i,\triangle}^\prime$ and $D^{(i)}_\triangle$ are invertible
and $B^{(i)}_\triangle$ is a full rank matrix. The following theorem
holds. \\

\begin{theorem} \label{teorema32}
For any $\lambda \in \hat{W}_\triangle(\Gamma)$ it holds that
\begin{equation} \label{precM}
\langle M \lambda, \lambda \rangle\; \leq \;\langle F \lambda, \lambda \rangle \;\leq \;C(1 + \log \frac{H}{h})^2 \langle M \lambda, \lambda \rangle
\end{equation}
where $C$ is a positive constant independent of $h_i$, $h_i/h_j$, $H_i$, $\lambda$ 
 and the jumps of $\rho_i$. Here and 
below, $\log(\frac{H}{h}):= \max_{i=1}^N \log(\frac{H_i}{h_i})$. \\
\end{theorem}
\begin{proof} We follow to the general abstract theory for 
FETI-DP methods developed in Theorem 6.35 of \cite{MR2104179}. This abstract
theory relies only on duality and linear algebra arguments, and     
properties such as that $B_\triangle$ is full rank, $P_\triangle$ is a
projection, $\tilde{S}$ is invertible and the subspace inclusion  (\ref{inclusion}). 
Using the same abstract arguments, the proof of the theorem 
follows by checking the Lemma \ref{lemma43}
and Lemma \ref{lemma45}, see below. The proof of   Lemma \ref{lemma45}
is not algebraic and it  depends on the problem. 
 The proofs of these two lemmas are presented
separately below.
\end{proof}

\begin{lemma} \label{lemma43}
For $u_\triangle \in {W}_\triangle(\Gamma^\prime)$ it follows that 
\begin{equation}
\langle \tilde{S} u_\triangle, u_\triangle \rangle\;\; \leq \;\; \langle {S}^\prime_\triangle u_\triangle, 
u_\triangle \rangle.
\end{equation}
\end{lemma}
\begin{proof} 
The proof follows from Lemma \ref{lemma41} and from 
\begin{equation}\label{430}
\langle \tilde{S} u_\triangle, u_\triangle \rangle = 
\min \langle \tilde{A}w, w \rangle\;\; \leq\;\;
 \min \langle \tilde{A}v, v \rangle\;\;  = \;\; \langle {S}^\prime_\triangle u_\triangle,u_\triangle \rangle
\end{equation}
where the minima are taken over  $w = (w_I, w_\Pi, w_\triangle) \in 
\tilde{W}(\Omega^\prime)$ such that $w_\triangle = u_\triangle$, 
and $v = (v_I, v_\Pi, v_\triangle) \in \tilde{W}(\Omega^\prime)$ such that
 $v_\Pi = 0$ and $v_\triangle = u_\triangle$.
\end{proof} 

\begin{lemma} \label{lemma45}
For any $u_\triangle \in {W}_\triangle(\Gamma^\prime)$  it holds that
\begin{equation} \label{pdeltain}
\parallel P_\triangle u_\triangle \parallel^2_{{S}_\triangle^\prime} \leq 
C(1 + \log \frac{H}{h})^2
\parallel u_\triangle \parallel^2_{\tilde{S}} 
\end{equation}
where $C$ is a positive constant independent of $h$ and  $ H$, $u_\triangle$
and the 
jumps of $\rho_i$.
\end{lemma}
The proof of this lemma is presented in the second part of next section.

\section{Refinement and Interpolations}

 In this section we introduce some technical tools to analyze 
the FETI-DP precondtioner. \\

\begin{definition}[The triangulation
  $\underline{\cal{T}}_{h}^i(\Omega_i)$] \label{definitri}   Let us introduce the triangulation $
\underline{\cal{T}}_{h}^i$, 
which is a refinement of ${\cal{T}}_h^i(\Omega_i)$ as follows.  
We refine an element  $\tau \in {\cal{T}}_h^i(\Omega_i)$ by 
considering four cases: \\
\begin{itemize}
\item {\normalfont No edge of $\tau$ belongs to $\partial \Omega_i$:}  Let 
$V_1$, $V_2$ and $V_3$ be the vertices of $\tau$ and let 
us denote $(\lambda_1(x),\lambda_2(x),\lambda_3(x))$ as the  
barycenter coordinates of a point $x \in \bar{\tau}$ such that 
$V_1 = (1,0,0)$, $V_2 =(0,1,0)$ and $V_3 = (0,0,1)$. Let 
$M_1 = (0,1/2,1/2)$, $M_2 = (1/2,0,1/2)$, $M_3 = (1/2,1/2,0)$ and
$C_1 = (2/3,1/6,1/6)$, $C_2 = (1/6,2/3,1/6)$, $C_3 = (1/6,1/6,2/3)$.
 The refinement of $\tau$ is defined by the triangles $C_1C_2C_3$, 
 $V_1M_3C_1$, $M_3 C_2 C_1$, $M_3 V_2 C_2$, $V_2M_1C_2$,
 $M_1V_3C_2$,
$M_1V_3C_3$, $V_3M_2C_3$, $M_2 C_1C_3$ and $M_2V_1 C_1$.  
See Figure \ref{fig:ref}, upper-left picture.

\item {\normalfont Only one edge of $\tau$ belongs to $\partial \Omega_i$:}  
 Let $V_1$, ${V}_2$ and ${V}_3$ be the vertices of $\tau$ and let us assume
that the edge opposed to $V_1$ belongs to $\partial \Omega_i$. Put
 $V_1 = (1,0,0)$, ${V}_2=(0,1,0)$, ${V}_3=(0,0,1)$ and let 
$M_2 = (1/2,0,1/2)$, $M_3 = (1/2,1/2,0)$,  
$C_1 = (1/3,1/3,1/3)$, $C_2 = (0,2/3,1/3)$ and $C_3 = (0,1/3,2/3)$. 
 The refinement of $\tau$
is defined by the triangles $V_1 M_3 C_1$,  $M_3 {V}_2 C_1$,
${V}_2C_2C_1$,
$C_2C_3C_1$, $C_3{V}_3C_1$, ${V}_3 M_2C_1$ and $ M_2 V_1
C_1$.  See Figure \ref{fig:ref}, upper-right picture.

\item {\normalfont Exactly two edges of $\tau$ belong to $\partial  \Omega_i$:} Let 
$C_1$, ${V}_2$ and ${V}_3$ be the vertices of $\tau$ and let 
us assume the edges opposed to $V_2$ and $V_3$ belong to $\partial \Omega_i$. 
We have $C_1 = (1,0,0)$, ${V}_2 = (0,1,0)$, ${V}_3 = (0,0,1)$, and define
 $M_1 = (0,1/2,1/2)$ and $C_2 = (1/2,0,1/2)$, $C_3 = (1/2,1/2,0)$.
The refinement of $\tau$ is defined by the triangles 
$C_1 C_3 C_2$, $C_3{V}_2 M_1$, $M_1{C}_3 C_2$ and  
$M_1V_3C_3$.  See Figure \ref{fig:ref}, lower-left picture.

\item {\normalfont All  three  edges of $\tau$ belong to $\partial \Omega_i$:}
 No refinement is needed,  let $C_1$, $C_2$ and $C_3$ be the 
vertices of $\tau$. \\
\end{itemize} 
\end{definition}

We distinguish  the above nodes by 
saying that they are nodes of type $C$, type $M$ and type $V$, respectively. 
It is easy to see that $\underline{\cal{T}}_{h}^i$,  the refinement
of ${\cal{T}}_{h}^i$,  is 
a geometrically conforming, shape regular and quasi-uniform triangulation  of $\Omega_i$.
Therefore, 
we denote by $\underline{W}_i(\Omega_i)$ as the 
space of piecewise linear and continuous functions on 
$\underline{\cal{T}}_{h}^i$.
 
\begin{definition}[The space
  $\underline{W}_i(\Omega_i^\prime)$] \label{definiunder} Let $\underline{W_i}(\bar{E}_{ji})$ be the space 
$\underline{W}_j(\Omega_j)$ restricted to $\bar{E}_{ji}$, for $j \in 
{\cal{J}}_H^{i,0}$. We define 
$\underline{W}_i(\Omega_i^\prime)$ by 
\begin{equation} \label{defiwi2}
\underline{W}_i({\Omega}^\prime_i) = \underline{W}_i({\Omega}_i) \times
\prod_{j \in {\cal{J}}_{H}^{i,0}} \underline{W}_i(\bar{E}_{ji}).
\end{equation} 
\end{definition}

A function $ \underline{u}_i \in \underline{W}_i({\Omega}^\prime_i)$ 
is represented by 
\[
\underline{u}_i = \{(\underline{u}_i)_i, \{(\underline{u}_i)_j\}_{j \in {\cal{J}}_H^{i,0}} \},
\]
where $(\underline{u}_i)_i := \underline{u}_{i\, | \overline{\Omega}_i}$ ($\underline{u}_i$ 
restricted to $\overline{\Omega}_i$) and 
$(\underline{u}_i)_j := \underline{u}_{i\, | \bar{E}_{ji}}$ ($\underline{u}_i$ restricted
to  $\bar{E}_{ji}$). Note that $(\underline{u}_i)_i$ and
$(\underline{u}_i)_j $ are continuous on $\overline{\Omega}_i$ and $\bar{E}_{ji}$, repectively. 
\\

\begin{figure}[htb]
\centering
{\psfig{figure=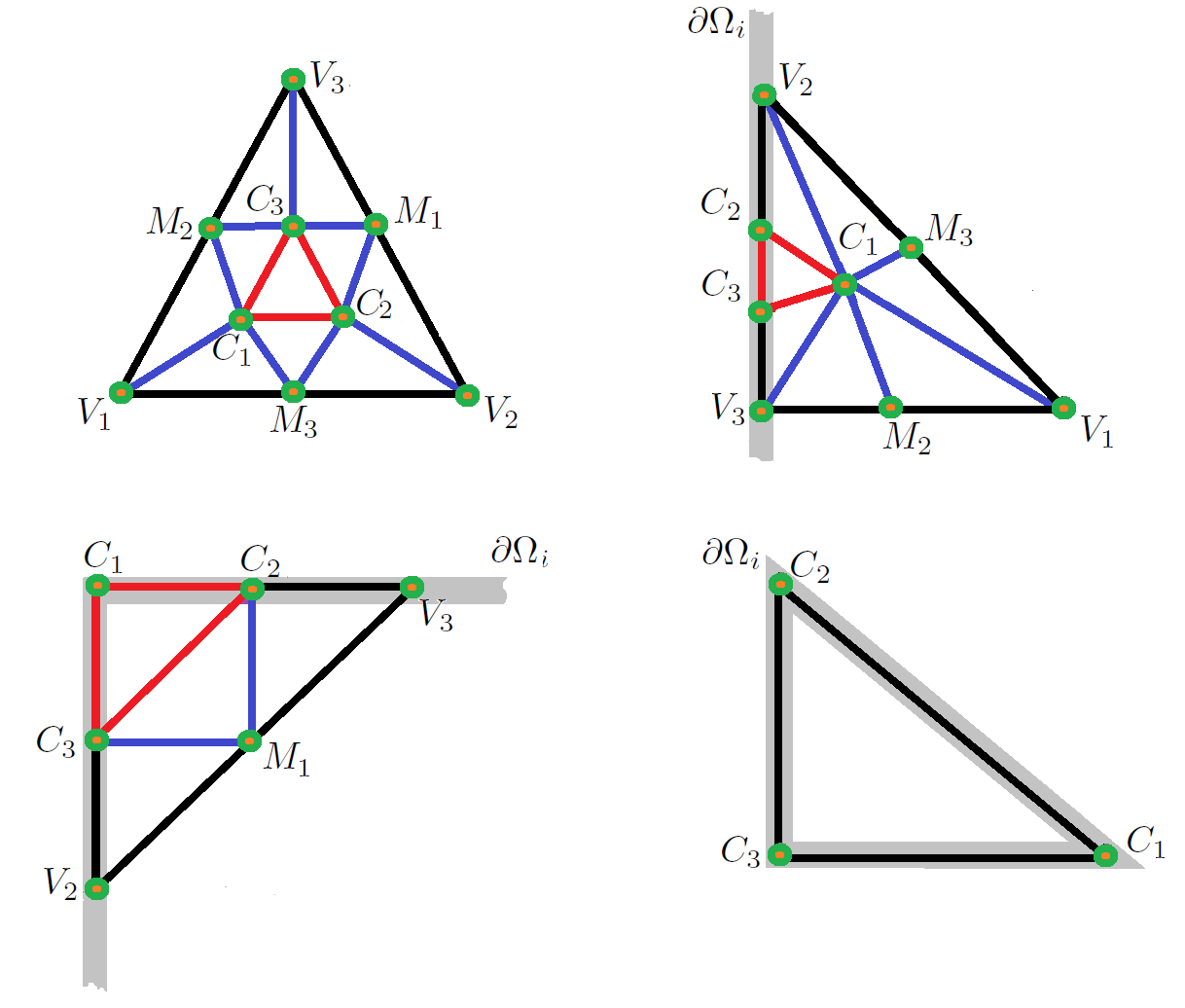,height=10cm,width=12cm,angle=0}}
\caption{Illustration of refinement of an element according to its relative 
position with respect to $\partial\Omega_i$.}
\label{fig:ref}
\end{figure}

\begin{definition}[The interpolators
$\underline{I}_{h}^i$] \label{definiunder3}
Given 
$u_i = \{({u}_i)_i, \{({u}_i)_j\}_{j \in {\cal{J}}_H^{i,0}} \} \in {W}_i(\Omega_i^\prime)$, we construct \[
\{(\underline{u}_i)_i,
\{(\underline{u}_i)_j\}_{j   \in {\cal{J}}_H^{i,0}} \} =
\underline{I}_{h}^i u_i
\in  \underline{W}_i(\Omega_i^\prime)~~~~\mbox{as follows:}
\]
\begin{itemize} 
\item Define $(\underline{u}_i)_i \in \underline{W}_i(\Omega_i)$ by assingning
values at the nodes of type $C,M,V \in \underline{\cal{T}}_h^i(\overline{\Omega}_i)$: 
\begin{itemize} 
\item Nodes of type $C$: define $(\underline{u}_i)_i(C) =
  ({u}_i)_i(C)$. 
\item Nodes of type $M$ (with $M$ shared by 
${\tau}_{-},{\tau}_{+} \in {\cal{T}}_h^i(\Omega_i)$: define  $(\underline{u}_i)_i(M) = 1/2(u_{\tau_{+}}(M) + u_{\tau_{-}}(M))$.
\item Nodes of type ${V} \notin \partial \Omega_i$ (with $V$
  shared by the elements  
$\tau_1, \cdots, \tau_p \in {\cal{T}}_h^i$): define 
$(\underline{u}_i)_i(V) = 1/p\;(u_{\tau_1}(V)+\cdots+
u_{\tau_p}(V))$.
\item Nodes of type ${V} \in \partial \Omega_i$ (with ${V}$ shared by two
  element edges $e_r, e_\ell \subset \partial \Omega_i$): 
define $(\underline{u}_i)_i({V}) =  1/2\;(u_{\tau_r}({V})+
u_{\tau_\ell}({V}))$, where the elements $\tau_r$ and $\tau_\ell$ share  
the edges $e_r$ and $e_\ell$, respectively. 
\end{itemize} 
\item  For $j \in {\cal{J}}_H^{i,0}$ define $(\underline{u}_i)_j \in
  \underline{W}_i(\overline{E}_{ji})$ by assingning
nodal values at the nodes of type $C,V$ on
$\underline{\cal{T}}_h^j(\bar{E}_{ji})$:  
\begin{itemize} 
\item Nodes of type $C$ define $(\underline{u}_i)_j(C) = ({u}_i)_j(C)$
\item Nodes of type ${V}$ define  $(\underline{u}_i)_j({V})$ according to 
\begin{itemize}
\item If $V\in \partial \bar{E}_{ji}$, 
 define $(\underline{u}_i)_j({V}) = \lim_{(x \in E_{ji})
  \rightarrow {V}} (u_i)_j(x)$.
\item If $V\not \in \partial \bar{E}_{ji}$, define
 $(\underline{u}_i)_j({V}) =
  1/2\;( u_{e_r} ({V})+ u_{e_\ell }({V}))$, where $e_r$ and
  $e_\ell$ are the two  edges of ${\cal{E}}_h^{j,i}$ sharing ${V}$. \\
\end{itemize}
\end{itemize}
\end{itemize} 
\end{definition}

  Note that from the definition of $\underline{u}_i =
 \underline{I}_{h}^i u_i$, the value of $(\underline{u}_i)_i$ 
  on $\partial \Omega_i$
 depends only on the value of $(u_i)_i$ on $\partial
 \Omega_i$. In addition, the values of $(u_i)_j$  depends only on the value of $(u_i)_j$ on
 $E_{ji}$ for any $j \in {\cal{J}}_{H}^{i,0}$. 
We note however that the value of $(\underline{u}_i)_i$ on an edge 
 $E_{ij}$ does not necessary depend only on the values of $(u_i)_i$ on $E_{ij}$ due to
the way that the nodal values at the subdomain corners are assigned.

  For a fixed edge $j \in{\cal{J}}_{H}^{i}$, we next modify
$\underline{I}_{h}^i$ (denoted by $\underline{I}_{{h}}^{i,j}$) 
so that $\underline{u}_i := \underline{I}_{{h}}^{i,j}  u_i$ 
on $E_{ij}$ depends only on the value of  $(u_i)_i$ on $E_{ij}$. We note
however that $ (\underline{u}_i)_i$ on $E_{ik}$, for $ k \in
{\cal{J}}_{H}^{i}$ and $k \neq j$, { does not depend 
 on the values of $(u_i)_i$ on $E_{ik}$}. The other properties described 
above for   $\underline{I}_{{h}}^{i}$ also holds for $\underline{I}_{{h}}^{i,j}$.  
We note that in the case of Figure \ref{fig:ref}, lower-left picture, the operator 
$\underline{I}_{{h}}^{i}$ would be enough in order to perform the analysis 
of the FETI-DP method. The operator  $\underline{I}_{{h}}^{i,j}$ is needed 
only in cases where two or more elements touch simultaneouly a corner of
$\Gamma_i$. \\ 

\begin{definition}[The interpolator
  $\underline{I}_{{h}}^{i,j}$] We introduce 
the interpolator $\underline{I}_{{h}}^{i,j} :
 {W}_i(\Omega_i^\prime) \rightarrow \underline{W}_i(\Omega_i^\prime)$, by modifying $\underline{I}_{{h}}^i$ 
only for the definition of $(\underline{u}_i)_i$ at the two endpoints
of $\partial E_{ij}$. For ${V}$ an endpoint of  $\partial
E_{ij}$ define $(\underline{u}_i)_i({V}) = \lim_{(x \in E_{ij})
  \rightarrow {V}} (u_i)_i(x)$. \\
\end{definition} 

 Let us define the following bilinear forms for the space
 $\underline{W}_i(\Omega_i^\prime)$
\begin{equation} \label{aprimei2}
\underline{a}^\prime_{i}(\underline{u}_i, \underline{v}_i)  = \underline{a}_i(\underline{u}_i , \underline{v}_i) + 
s_{i,\partial}(\underline{u}_i,\underline{v}_i) + p_{i,\partial}(\underline{u}_i,\underline{v}_i)
\end{equation}
and the semi-norm
\[
\underline{d}_i(\underline{u}_i,\underline{u}_i) = 
\underline{a}_i(\underline{u}_i,\underline{u}_i) + 
{p}_{i,\partial}(\underline{u}_i,\underline{u}_i)
\]
where 
\[
\underline{a}_i(
\underline{u}_i,\underline{u}_i): = \int_{\Omega_i} \rho_i \nabla
(\underline{u}_i)_i \cdot  \nabla (\underline{u}_i)_i \,dx.
\]

 The following lemma holds. From here on, in order to avoid the proliferation  of constants, 
we use the notation $A\preceq B$ to express the fact that there is a constant $C$  independent 
of $h$ such that 
$A\leq C A$. Similarly for the  symbol $A\asymp B$ which means that 
$A\preceq B$ and also $B\preceq A$.
\begin{lemma} \label{underlineII}   
For $u_i \in W_i(\Omega^\prime_i)$ and let $\underline{u}_i = 
\underline{I}_{{h}}^i u_i$  or $\underline{u}_i = 
\underline{I}_{{h}}^{i,j} u_i$. Then 
\begin{equation}\label{eq:firstinequality}
\underline{a}_i((\underline{u}_i)_i,(\underline{u}_i)_i) \asymp 
{a}_i(({u}_i)_i,({u}_i)_i),
\end{equation}
and 
\begin{equation}\label{eq:secondinequality}
{p}_{\partial, i}(\underline{u}_i,\underline{u}_i) 
\preceq 
{p}_{\partial, i}({u}_i,{u}_i)) + 
{a}_i(({u}_i)_i,({u}_i)_i),
\end{equation}
\begin{equation}\label{5.2_1}
{p}_{\partial, i}({u}_i,{u}_i)  \preceq 
{p}_{\partial, i}(\underline{u}_i,\underline{u}_i) 
+ 
\underline{a}_i((\underline{u}_i)_i,(\underline{u}_i)_i).
\end{equation}
Additionally, if $\underline{u}_i = 
\underline{I}_{{h}}^{i,j} u_i$, then
\begin{equation}\label{eq:fourthinequality}
\|(u_i)_i - (u_i)_j\|_{L^2(E_{ij})} \asymp 
\|(\underline{u}_i)_i - (\underline{u}_i)_j\|_{L^2(E_{ij})}.
\end{equation}
\end{lemma}
The proof of this and next lemma are presented in the Appendix.

\begin{definition} \label{defininounder} [The interpolator
  $I_h^i$] We now introduce the interpolator $I_h^i: \underline{W}_i(\Omega_i^\prime)\rightarrow 
W_i(\Omega_i^\prime)$ as follows. Given $\underline{u}_i =  \{(\underline{u}_i)_i,
\{(\underline{u}_i)_j\}_{j   \in {\cal{J}}_H^{i,0}} \} \in
\underline{W}_i(\Omega_i^\prime)$ we construct 
\[ 
 \{({u}_i)_i, \{({u}_i)_j\}_{j \in {\cal{J}}_H^{i,0}} \} := I_h^i\underline{u}_i
\in
{W}_i(\Omega_i^\prime) ~~~\mbox{as follows:}
\]
\begin{itemize} 
\item For $\tau \in {\cal{T}}_h^i$, let $C_1$, $C_2$, $C_3$
be the nodes of type $C$  in $\underline{\cal{T}}_{h}^i$ on
$\bar{\tau}$.  We define $(u_i)_i$ on $\tau$ as the linear extrapolation on $\tau$ of the linear function
defined by the nodal values $(\underline{u}_i)_i(C_1),
(\underline{u}_i)_i(C_2)$ and $(\underline{u}_i)_i(C_3)$. 
\item For   $j \in {\cal{J}}_H^{i,0}$ and $e \in {\cal{E}}_h^{j,i}$,
let $C_1 $and  $C_2$
be the nodes of type C on element edge $e$. We define $(u_i)_j$ on element 
edge  $e$ as the
linear extrapolation on $e$ of the 
linear function defined by the nodal values $(\underline{u}_i)_j(C_1)$ and 
$(\underline{u}_i)_j(C_2)$. \\
\end{itemize}
\end{definition} 

The following lemma holds.\\
\begin{lemma} \label{underlineI} Let $\underline{u}_i \in
  \underline{W}_i(\Omega_i^\prime)$. Then 
\begin{equation}\label{eq:lem571}
a_i(I_h^{i} \underline{u}_i, I_h^{i} \underline{u}_i)    \preceq 
\underline{a}_i( \underline{u}_i, \underline{u}_i) 
\end{equation}
and 
\begin{equation}\label{eq:lem572}
p_{i,\partial} (I_h^{i} \underline{u}_i,I_h^{i} \underline{u}_i) \preceq 
p_{i,\partial} (\underline{u}_i,\underline{u}_i). 
\end{equation}
In case $\underline{u}_i = \underline{I}_{{h}}^i u_i$ or
$\underline{u}_i = \underline{I}_{{h}}^{i,j} u_i$, we have 
then  
\begin{equation}\label{eq:lem573}
 I_h^i \underline{u}_i = u_i.
\end{equation}
\end{lemma}

Let us also introduce $\underline{\cal{H}}_i \underline{u}_{i,\Gamma^\prime} \in 
\underline{W}_i({\Omega}^\prime_i)$
 as  the standard discrete harmonic function 
of $\underline{u}_{i,\Gamma^\prime} \in \underline{W}_i({\Gamma}_i^\prime)$, i.e.,
$\underline{\cal{H}}_i \underline{u}_{i,\Gamma^\prime} = \underline{u}_{i,\Gamma^\prime}$ on $\Gamma_i^\prime$
and discrete harmonic in ${\Omega}_i$ 
in the sense of $\underline{a}_i(\cdot,\cdot)$, see (\ref{aprimei2}), with Dirichlet data
on $\Gamma_i$. We note that the extensions 
$\underline{\cal{H}}_i$ and ${\cal{H}}^\prime_i$ differ from each other 
not only because they  are defined in  different space 
$\underline{W}_i(\Omega_i^\prime)$ and ${W}_i(\Omega_i^\prime)$, respectively,
but also because 
 $\underline{\cal{H}}_i \underline{u}_{i,\Gamma^\prime}$ at the interior nodes 
of $\underline{\cal{T}}_h^i$ 
depends only on the nodal values of $\underline{u}_{i,\Gamma^\prime}$ on $\Gamma_i$, 
while 
${\cal{H}}^\prime_i u_{i,\Gamma^\prime}$ depends on the  nodal values of  
$u_{i,\Gamma^\prime}$ on ${\Gamma}_i^\prime$. \\

The following lemma shows 
the equivalence (in the energy form defined 
by $d_i(\cdot,\cdot)$)  between discrete harmonic functions in the 
sense of $\underline{\cal{H}}_i$ and in the sense of ${\cal{H}}^\prime_i$; for the 
proof see  \cite{MR2372024}. This 
equivalence allows us to take advantages of all the discrete Sobolev results known for $\underline{\cal{H}}_i$ discrete harmonic extensions.\\

\begin{lemma} \label{lemmaequiv}
Let $u_i  \in {W}_i(\Omega_i^\prime)$ and $\underline{u}_i \in  
\underline{W}_i(\Omega_i^\prime)$ defined by $\underline{u}_i = 
\underline{I}_{{h}}^i u_i$ or by $\underline{u}_i =
\underline{I}_h^{i,j}  u_i$. 
Then 
\begin{equation} \label{equivd} 
\underline{d}_i(\underline{\cal{H}}_i \underline{u}_{i}, \underline{\cal{H}}_i \underline{u}_{i}) \asymp 
d_i({\cal{H}}^\prime_i u_{i}, {\cal{H}}^\prime_i u_{i})
\end{equation}
where $C$ is a positive constant independent of $\delta$, $h_i, H_i$, $\rho_i$ and $u_{i}$. 
\end{lemma} 
\begin{proof}
First note by construction that $I_h^i \underline{\cal{H}}_i
\underline{I}_h^i u_i =I_h^i \underline{I}_h^i  u_i=u_i$ on 
$\Gamma_i^\prime$. Using Lemma \ref{lemmac}, a 
minimum ${\cal{H}}_i^\prime$-energy argument and  Lemma \ref{lemmac}
again, we obtain 
\[
d_i({\cal{H}}^\prime_i u_{i}, {\cal{H}}^\prime_i u_{i}) \asymp
a_i^\prime({\cal{H}}^\prime_i u_{i}, {\cal{H}}^\prime_i u_{i}) 
\leq 
a_i^\prime(I_h^i \underline{\cal{H}}_i \underline{u}_i,I^i_h
\underline{\cal{H}}_i \underline{u}_i) \asymp 
d_i(I^i_h \underline{\cal{H}}_i \underline{u}_i,I_h^i
\underline{\cal{H}}_i \underline{u}_i).
\]
By Lemma \ref{underlineI}  we have 
\[
d_i(I^i_h \underline{\cal{H}}_i \underline{u}_i,I_h^i
\underline{\cal{H}}_i \underline{u}_i) \preceq 
\underline{d}_i (\underline{\cal{H}}_i \underline{u}_i,
\underline{\cal{H}}_i \underline{u}_i).
\]

The proof of the left  inequality of (\ref{equivd}) is complete. The
proof using the operator $\underline{I}_h^{i,j}$ instead of 
$\underline{I}_h^i$ is similar. \\

 Now let as prove the left inequality of  (\ref{equivd}). Note that  
$\underline{\cal{H}}_i \underline{u}_i = \underline{I}_h^i {\cal{H}}^\prime_i
u_{i} = \underline{u}_i 
$ on $\Gamma_i^\prime$.  Using a minimal ${\cal{H}}^\prime_i$-energy argument 
 and Lemma \ref{underlineII} 
we obtain 
\[
\underline{d}_i(\underline{\cal{H}}_i \underline{u}_i, 
\underline{\cal{H}}_i \underline{u}_i)\asymp 
\underline{a}_i (\underline{\cal{H}}_i,
\underline{\cal{H}}_i \underline{u}_i)  \leq 
\underline{a}_i (\underline{I}_h^i {\cal{H}}^\prime_i u_{i}, \underline{I}_h^i
{\cal{H}}^\prime_i u_{i}) \asymp 
{a}_i ({\cal{H}}^\prime_i u_{i}, {\cal{H}}^\prime_i u_{i}),
\]
and again Lemma \ref{underlineII} we have 
\[ 
p_{i,\partial}(\underline{\cal{H}}_i \underline{u}_i,
\underline{\cal{H}}_i \underline{u}_i) =  
p_{i,\partial}(\underline{I}_h^i
{\cal{H}}^\prime_i u_{i}, \underline{I}_h^i
{\cal{H}}^\prime_i u_{i}) \preceq 
d_i({\cal{H}}^\prime_i u_{i}, {\cal{H}}^\prime_i u_{i}),
\]
therefore, the left  inequality of (\ref{equivd}) follows. The proof
with $\underline{I}_h^{i,j}$ is similar. 
\end{proof}

We are now in position to prove Lemma \ref{lemma45}.

{\bf Proof of Lemma \ref{lemma45}.}\begin{proof}
 We first consider the case when the 
edges $E_{ij}$ are made by a single interval only. 
Let $u_\triangle \in {W}_\triangle(\Gamma^\prime)$ and let 
$u = (u_\Pi,u_\triangle) \in \tilde{W}(\Gamma^\prime)$ be the solution of 
\begin{equation} \label{equivS}
\langle \tilde{S}u_\triangle,u_\triangle \rangle \;\;=\;\; 
\min \langle {S}^\prime w,w \rangle =: \langle {S}^\prime u,u \rangle,
\end{equation} 
where the minimum is taken over $w = (w_\Pi, w_\triangle) \in 
\tilde{W}(\Gamma^\prime)$ such that   $w_\Pi \in \hat{W}_\Pi(\Gamma^\prime)$
and $w_\triangle = u_\triangle$. Here $S^\prime=    \mbox{diag}\{S_i^\prime\}_{i=1}^N$ where 
$S_i^\prime$ is defined in (\ref{Si}). The problem has a unique solution, see (\ref{430}).
Hence, we can 
replace $\|u_\triangle\|_{\tilde{S}}$          in (\ref{pdeltain})
   by $\|u          \|_{S^\prime}$. \\

Let us represent the $u$ defined above 
as $\{u_i\}_{i=1}^N \in  W(\Gamma^\prime)$
where $u_i\in W_i({\Gamma}_i^\prime)$. Let $I_{E_{ij}} (u_i)_i$ be 
the linear function on $\bar{E}_{ij}$ defined  by 
the values of $(u_i)_i$ at $x \in \partial {E}_{ij}$, and let 
$I_{E_{ji}} (u_i)_j$ be the linear function on $\bar{E}_{ji}$ 
defined by the 
values of $(u_i)_j$ at $x \in \partial {E}_{ji}$.  
Let $\hat{u} := \{\hat{u}_i\}^N_{i=1}$ where 
$\hat{u}_i \in W_i({\Gamma}_i^\prime)$ is  defined by 
$$
(\hat{u}_i)_i = I_{E_{ij}} (u_i)_i \;\; \mbox{on} \;\; \bar{E}_{ijh} \;\; \mbox{for all}\;\;
j \in \,{\mathcal{J}}_H^{i,0}
$$
and
$$
(\hat{u}_i)_j = I_{E_{ji}} (u_i)_j \;\; \mbox{on} \;\; \bar{E}_{jih} \;\;
\mbox{for all}\;\; j \in \,{\mathcal{J}}_H^{i,0}.
$$
Note that $\hat{u} \in \hat{W}(\Gamma^\prime)$, therefore, let us 
represent $\hat{u} = (\hat{u}_\Pi,\hat{u}_\triangle)$ where 
$B_\triangle \hat{u}_\triangle = 0$. Using this we have, see (\ref{Ptri}),
$$
P_\triangle u_\triangle \equiv B^T_{D, \triangle} B_\triangle u_\triangle = B^T_{D, \triangle} B_\triangle(u_\triangle - \hat{u}_\triangle) = P_\triangle
(u_\triangle - \hat{u}_\triangle). 
$$

Note that $u - \hat{u} =0 $ at the $\Pi$-nodes, hence,  
let us define $v \in {W}(\Gamma^\prime)$ to be equal to 
$P_\triangle (u_\triangle - \hat{u}_\triangle)$ at the 
$\triangle$-nodes and equal to zero  at the $\Pi$-nodes. Let us 
represent  $v = \{v_i\}_{i=1}^N$ where $v_i \in W_i({\Gamma}_i^\prime)$. We 
have 
\begin{equation} \label{Pdeltaw}
\parallel P_\triangle u_\triangle \parallel^2_{{S}^\prime_\Delta} =
\parallel v \parallel^2_{{S}^\prime} = \sum^N_{i=1} \parallel 
v_i\parallel^2_{{S}^\prime_i}
\end{equation} \\
in view of the definition of $S_{i,\triangle}^\prime$, $S^\prime_\Delta$ and $S^\prime$,
see (\ref{Stri}), (\ref{Si}) and (\ref{tildeS}), hence, to prove the lemma
it remains to show that 
\begin{equation} \label{referS} 
\sum^N_{i=1} \parallel v_i \parallel^2_{S_i^\prime} \leq 
C (1 + \log{H/h})^2 \|u\|^2_{S^\prime}
\end{equation}
since by (\ref{equivS}) we obtain (\ref{pdeltain}). By Lemma \ref{lemmac}  we need to show 
\begin{equation} \label{viSi}
\sum^N_{i=1} 
{d}_i({\mathcal{H}}^\prime v_i, {\mathcal{H}}^\prime v_i) \leq C (1 + \log{H/h})^2 \sum^N_{i=1} 
{d}_i({\mathcal{H}}^\prime u_i, {\mathcal{H}}^\prime u_i).
\end{equation} 
Define $\underline{v}_i = \underline{I}_h^i v_i$.  Note that because
$v_i$ vanishes at the $\Pi$- nodes, $\underline{v}_i$ also vanishes 
at the $\Pi$-nodes. Using Lemma \ref{lemmaequiv} we obtain
\[ 
{d}_i({\mathcal{H}}^\prime v_i, {\mathcal{H}}^\prime v_i) \asymp
\underline{d}_i(\underline{\cal{H}}_i \underline{v}_i,
\underline{\cal{H}}_i \underline{v}_i).
\]
From now on, let us denote $\underline{v}_i = \underline{\cal{H}}_i
\underline{v}_i$  and  $\underline{v}_j =
\underline{\cal{H}}_j \underline{v}_j$
and so, 
\begin{equation} \label{hatdivi}
\underline{d}_i(\underline{v}_i, \underline{v}_i) = \rho_i \parallel \nabla (\underline{v}_i)_i \parallel^2_{L^2(\Omega_i)} + \sum_{j \in \,{\mathcal{J}}_H^i} 
\frac{\rho_i \delta}{l_{ij} h_{e}} \parallel (\underline{v}_i)_i- (\underline{v}_i)_j \parallel^2_{L^2(E_{ij})}.
\end{equation}

  We first estimate the first term of (\ref{hatdivi}). We have
\begin{equation} \label{Hhalf} 
\parallel \nabla (\underline{v}_i)_i \parallel^2_{L^2(\Omega_i)} 
\leq C \sum_{ j \in  \,{\mathcal{J}}_H^{i,0}} \parallel 
(\underline{v}_i)_i \parallel^2_{H^{1/2}_{00}(E_{ij})}
\end{equation}
by the well-known estimate, see \cite{MR2104179}, and the fact that 
$(\underline{v}_i)_i= 0$ at corners of $\Gamma_i$. Note that 
(\ref{Hhalf}) is also valid for subdomains $\Omega_i$ 
which 
intersect  $\partial \Omega$ by edges since we use the obvious inequality 
\[
\parallel \nabla (\underline{v}_i)_i \parallel^2_{L^2(\Omega_i)} \leq 
\parallel \nabla (\tilde{\underline{v}}_i)_i \parallel^2_{L^2(\Omega_i)}
\leq  C \sum_{ j \in 
 \,{\mathcal{J}}_H^{i,0}} \parallel (\underline{v}_i)_i \parallel^2_{H^{1/2}_{00}(E_{ij})}
\]
where $(\tilde{\underline{v}}_i)_i$ is the $\underline{\cal{H}}_i$ discrete harmonic extension
on $\Omega_i$ with $(\tilde{\underline{v}}_i)_i = (\underline{v}_i)_i$ on edges $\bar{E}_{ij}$
 for $ j \in {\mathcal{J}}_{H}^{i,0}$, 
and $(\tilde{v}_i)_i = 0$ on edges $\bar{E}_{ij}$ for $j \in
{\mathcal{J}}_{H}^{i,\partial}$. For the case 
$E_{ij}$ such that $j \in  {\mathcal{J}}_{H}^{i,0}$, define
$\underline{u}_i = \underline{I}_h^{i,j} u_i$ and $\underline{u}_j =
\underline{I}_h^{j,i} u_j$. Note also  that $\hat{\underline{u}}_i =
I_h^{i,j} \hat{u}_i = \hat{u}_i $ on  $\bar{E}_{ij}$ and $\bar{E}_{ji}$, and also
$\hat{\underline{u}}_j = I_h^{j,i} \hat{u}_j =  \hat{u}_j$. 
We use (\ref{Pdelta1}) to get
\begin{eqnarray} \label{virhoi}
&&\rho_i \parallel (\underline{v}_i)_i \parallel^2_{H^{1/2}_{00}(E_{ij})} = 
\frac{\rho_i \rho_j^{2\beta}}{(\rho^{\beta}_i + \rho^{\beta}_j)^2} \parallel (\underline{u}_i -
 \hat{\underline{u}}_i)_i - (\underline{u}_j - \hat{\underline{u}}_j)_i \parallel^2_{H^{1/2}_{00} (E_{ij})} \leq
 \nonumber \\
&\leq& 3\;\{
  \rho_i \parallel(\underline{u}_i - \hat{\underline{u}}_i)_i \parallel^2_{H^{1/2}_{00}(E_{ij})} +  \rho_j
 \parallel(\underline{u}_j - \hat{\underline{u}}_j)_j \parallel^2_{H^{1/2}_{00}(E_{ji})} + \\
&+& 
\frac{\rho_i \rho_j^{2\beta}}{(\rho^{\beta}_i + \rho^{\beta}_j)^2} 
\parallel(\underline{u}_j - \hat{\underline{u}}_j)_i - (\underline{u}_j-\hat{\underline{u}}_j)_j
 \parallel^2_{H^{1/2}_{00}(E_{ij})}
\;\}, \nonumber
\end{eqnarray}
where we have used that 
$\frac{\rho_i \rho_j^{2\beta}}{(\rho^{\beta}_i + \rho^{\beta}_j)^2} \leq 
\min\{\rho_i,\rho_j\}$ if $\beta \in [1/2,\infty)$, see \cite{MR1469678}. 

 Following the same steps of the proof of Lemma 4.5 in  \cite{MR3033016}  (see 
there (4.49)-(4.51)), we can bound 
\begin{equation} \label{virhoi2}
\rho_i \parallel (\underline{v}_i)_i
\parallel^2_{H^{1/2}_{00}(E_{ij})} 
\leq   C(1 + \log \frac{H}{h})^2 \{
\underline{d}_i(\underline{\cal{H}}_i \underline{u}_i,\underline{\cal{H}}_i \underline{u}_i) + 
\underline{d}_j(\underline{\cal{H}}_j \underline{u}_j,\underline{\cal{H}}_j \underline{u}_j)\}
\end{equation}
and using Lemma \ref{lemmaequiv} we obtain 
\begin{equation} \label{virhoi3}
\rho_i \parallel (\underline{v}_i)_i
\parallel^2_{H^{1/2}_{00}(E_{ij})} 
\leq   C(1 + \log \frac{H}{h})^2 \{
{d}_i({\cal{H}}^\prime_i {u}_i, {\cal{H}}^\prime_i {u}_i) + 
{d}_j({\cal{H}}^\prime_j {u}_j, {\cal{H}}^\prime_j {u}_j)\}. 
 \end{equation}

 It remains to estimate the second term of the right-hand side of
 (\ref{hatdivi}). The case $E_{ij}$ where $j  \in
 {\mathcal{J}}_{H}^{i,\partial}$ is trivial. For the 
case $E_{ij}$ such that $ j \in {\mathcal{J}}_{H}^{i,0}$ 
using  (\ref{Pdelta1}) - (\ref{Pdelta2}), and  similar
arguments as in the proof of Lemma 4.5 in \cite{MR3033016} (see there (4.45)-(4.51)), we obtain 
\begin{eqnarray}
& & 
\frac{\delta\rho_i }{l_{ij} h_e} \parallel (\underline{v}_i)_i - (\underline{v}_i)_j \parallel^2_{L^2(E_{ij})} =
 \frac{\rho_i \rho_j^{2\gamma}}{(\rho^{\gamma}_i + \rho^{\gamma}_j)^2}\times \nonumber \\
&\leq &  C(1 + \log \frac{H}{h})^2 \{
\underline{d}_i(\underline{\cal{H}}_i \underline{u}_i,\underline{\cal{H}}_i \underline{u}_i) + 
\underline{d}_j(\underline{\cal{H}}_j \underline{u}_j,\underline{\cal{H}}_j \underline{u}_j)\}.
\end{eqnarray}
Using Lemma \ref{lemmaequiv},  we obtain 
\begin{eqnarray}\nonumber
\frac{\delta\rho_i }{l_{ij} h_e} \parallel (\underline{v}_i)_i - (\underline{v}_i)_j \parallel^2_{L^2(E_{ij})}  &\leq&
 C(1 + \log \frac{H}{h})^2 \times \\
&&\{
{d}_i({\cal{H}}^\prime_i {u}_i, {\cal{H}}^\prime_i {u}_i) + 
{d}_j({\cal{H}}^\prime_j {u}_j, {\cal{H}}^\prime_j {u}_j)\}. 
 \label{virhoi34}
\end{eqnarray}
Using the inequalities (\ref{virhoi3}) and (\ref{virhoi34}) in 
(\ref{hatdivi}), summing the resulting inequality for $i$ from 1 to $N$ and  noting that
the number of edges of each subdomain can be bounded independently of $N$, 
we obtain  (\ref{viSi}) and (\ref{referS}).

  The proof also works with minor modifications
for the case when $E_{ij}$ is a continuous curve
of intervals. For that, we should consider 
discrete Sobolev tools for non straight edges, 
see for instance \cite{MR2421044}, and interpret  
$I_{E_{ij}} (u_i)_i$ and $I_{E_{ji}}(u_i)_j$ 
as the linear function with respect to parametrized path
on the edge defined by the nodal value of  $(u_i)_i$ or $(u_i)_j$ 
at $x \in \partial E_{ij}$ and $\partial E_{ji}$. 
\end{proof}

\section{Numerical experiments} \label{sec:num}
In this section, we present numerical 
results for solving the linear system 
(\ref{ProjF0}) with the left preconditioner (\ref{Minv}). 
 We  show that the lower and upper bounds of Theorem 
\ref{teorema32} are reflected in  the numerical tests. In particular
we show that the constant $C$ in (\ref{precM}) does not depend on 
$h_i$, $H_i$, and the jumps of $\rho_i$.

 We consider the domain $\Omega = (0,1)^2$ and divide 
 into $N = M \times M$  squares subdomains $\Omega_i$
of size $H = 1/M$. 
Inside each subdomain
$\Omega_i$  we generate a structured triangulation with
$n=m\times m$ subintervals in each coordinate direction 
and apply the discretization presented in Section \ref{discretization} 
with  penalty term $\delta = 10$. In the numerical 
experiments we use a  red  and black checkerboard type of
subdomain partition, where the most bottom-left subdomain has
a black color. We solve the second order
elliptic problem  $ - \mbox{div} (\rho(x)\nabla u^{*}_{ex}(x)) = 1$
in $\Omega$ with homogeneous Dirichlet boundary conditions $u^* = 0$. 
In the numerical experiments, we run PCG until the
$l_2$ initial residual is reduced by a factor of  $10^{10}$. 

\begin{table}[htb] 
\centering 
\small 
\caption{ Number of iterations, condition numbers 
 (in parenthesis) for different sizes of coarse and local problems 
and 
with constant coefficient $\rho=1$. Here $\beta = 1$, see (\ref{diagscal}).}
\begin{tabular}{|l|c|c|c|c|}\hline
 {$M$ $\downarrow$ $m \to$}   & $  4$ & $  8$ & $  16$\\ \hline
 $4 $ &13 (2.28)& 13 (2.84)& 13 (3.61) \\
 $8 $ & 15 (2.50)& 17 (3.16)& 18 (4.01)\\
 $16 $  &15 (2.59)& 17 (3.28)& 20 (4.16)\\ \hline
\end{tabular}
\label{tab:FETIDPTest1}
\end{table}

In the first test of experiments we consider the constant coefficient case $\rho=1$. We consider 
different values of $N=M\times M$ coarse partitions and different values of local 
refinements $n=m\times m$. Table \ref{tab:FETIDPTest1} lists the
number of PCG iterations and in parenthesis the condition number
estimate of the preconditioned system. 
As expected from the analysis, 
 the condition numbers appear to be independent of the number of subdomains
and grow by a two-logarithmically factor when the size of the local problems 
increases. As expected from the theory, the lower bounds estimates are always very 
closed to one, therefore, we do not show in the tables. 
\begin{table}[htb] 
\centering 
\small 
\caption{
 Number of iterations, condition numbers 
 (in parenthesis) for different sizes of coarse and local problems 
and 
with constant coefficient $\rho=1$ in the black substructures and $\rho=1000$ 
in the red substructures. Here $\beta = 1$, see (\ref{diagscal}).}
\begin{tabular}{|l|c|c|c|c|}\hline
 {$M$ $\downarrow$ $m \to$}   & $  4$ & $  8$ & $  16$\\ \hline
 $4 $ & 5 (1.10)& 5 (1.10)& 5 (1.10)\\
 $8 $ & 6 (1.10)& 6 (1.12)& 6 (1.16)\\
 $16 $ & 7 (1.29)& 8 (1.42)& 8 (1.55)\\ \hline
\end{tabular}
\label{tab:FETIDPTest1mu1000}
\end{table}
\begin{table}[h] 
\centering 
\small 
\caption{Number of  iterations  
and condition numbers (in parenthesis) for different values
of the coefficient $\rho_r$ on
the red substructures and local meshes with $n=m\times m$. On black substructures the 
coefficient $\rho_b = 1$  is kept fixed. 
The substructure partition is also kept fixed to $N=M\times M=8 \times 8$.
 Here $\beta = 1$, see (\ref{diagscal}). }
\label{tab:FETIDPTest2}
\begin{tabular}{|c|c|c|c|c|}\hline
 $\rho_r \downarrow$ $m \to$ & $2$ & $ 4$ & $8$ & $16$ \\ \hline
 $10000 $ & 5 (1.10)& 5 (1.09)& 5 (1.09)& 5 (1.09)\\
 $1000 $ & 6 (1.10)& 6 (1.10)& 6 (1.12)& 6 (1.16)\\
 $100 $ & 7 (1.21)& 7 (1.35)& 8 (1.50)& 9 (1.66)\\
 $10 $ & 10 (1.50)& 11 (1.79)& 13 (2.15)& 15 (2.55)\\
 $1 $ & 12 (1.96)& 15 (2.50)& 17 (3.16)& 18 (4.01)\\
 $0.1 $ & 10 (1.51)& 12 (1.82)& 13 (2.18)& 15 (2.59)\\
 $0.01 $ & 7 (1.27)& 8 (1.44)& 9 (1.62)& 10 (1.80)\\
 $0.001 $ & 6 (1.10)& 6 (1.14)& 6 (1.21)& 6 (1.28)\\
 $0.0001 $ & 5 (1.10)& 5 (1.09)& 5 (1.09)& 5 (1.09)\\
 \hline\hline 
\end{tabular} 
\end{table}

 We now consider the discontinuous coefficients case where we set 
 $\rho_b = 1$ on the black  substructures and we vary $\rho_r$ on the 
red  substructures.
We first consider
different values of $N=M\times M$ coarse partitions and different values of local 
refinements $n=m\times m$ while we keep $\rho_r=1000$. 
The results are shown in Table \ref{tab:FETIDPTest1mu1000} and are similar to the 
previous test for continuous coefficient.

For the next experiment
 the substructures partition is kept fixed to
$8 \times 8$. Table \ref{tab:FETIDPTest2} lists the results on runs for different
values of $\rho_r$ and for different levels  of refinements.
In Table \ref{tab:FETIDPTest3} and Table  \ref{tab:FETIDPTest4} we 
repeat the test with two diferent values of $\beta$, see (\ref{diagscal}).
 The performance of the preconditioner is  
robust with respect to the 
coefficients and $h$ as predicted.
\begin{table}[htb] 
\centering 
\small 
\caption{Number of  iterations  
and condition numbers (in parenthesis) for different values
of the coefficient $\rho_r$ and local meshes with $n=m\times m$ on
the red substructures. On black substructures the 
coefficient $\rho_b = 1$  is kept fixed. 
The substructure partition is also kept fixed to $N=M\times M=8\times 8$. 
 Here $\beta = 0.5$, see (\ref{diagscal}).}
\label{tab:FETIDPTest3}
\begin{tabular}{|c|c|c|c|c|}\hline
 $\rho_r \downarrow$ $m \to$ & $2$ & $ 4$ & $8$ & $16$ \\ \hline
 $1000 $ & 20 (4.21)& 24 (5.37)& 25 (6.58)& 28 (9.68)\\
 $1 $ & 12 (1.96)& 15 (2.50)& 17 (3.16)& 18 (4.01)\\
 $0.001 $ & 20 (4.21)& 24 (5.39)& 25 (6.55)& 27 (9.54)\\
 \hline\hline 
\end{tabular} 
\end{table}

\begin{table}[htb] 
\centering 
\small 
\caption{Number of  iterations  
and condition numbers (in parenthesis) for different values
of the coefficient $\rho_r$ and local meshes with $n=m\times m$ on
the red substructures. On black substructures the 
coefficient $\rho_b = 1$  is kept fixed. 
The substructure partition is also kept fixed to $N=M\times M=8\times 8$. 
 Here $\beta = 10$, see (\ref{diagscal}).}
\label{tab:FETIDPTest4}
\begin{tabular}{|c|c|c|c|c|}\hline
 $\rho_r \downarrow$ $m \to$ & $2$ & $ 4$ & $8$ & $16$ \\ \hline
 $1000 $ & 6 (1.10)& 6 (1.10)& 6 (1.12)& 6 (1.16)\\
 $1 $ & 12 (1.96)& 15 (2.50)& 17 (3.16)& 18 (4.01)\\
 $0.001 $ & 6 (1.10)& 6 (1.14)& 6 (1.21)& 6 (1.29)\\ \hline\hline 
\end{tabular} 
\end{table}

\appendix 

\section{Proof of Lemma \ref{underlineII}}
We first prove the right hand side of the first inequality,  the inequality  (\ref{eq:firstinequality}). That is, we prove that there exist 
a constant $C$ such that, for all $u_i\in W_i(\Omega_i^\prime)$ we have
\begin{equation}\label{A.1}
\underline{a}_i((\underline{u}_i)_i,(\underline{u}_i)_i)\leq Ca_i(u_i,u_i).
\end{equation}
First, note that,
\begin{equation}
\underline{a}_i((\underline{u}_i)_i,(\underline{u}_i)_i)=
\sum_{\tau\in \mathcal{T}^i_h}\rho_i \int_{\tau} | \nabla(\underline{u}_i)_i|^2dx.
\end{equation}
We consider the cases of refined mesh $\underline{\mathcal{T}}^i_h$ listed 
in Definition \ref{definitri}  and illustrated in Figure \ref{fig:ref}. 

{\bf First case (Figure \ref{fig:ref}, upper-left picture)}.
For the first case, that is, $\tau\in \mathcal{T}^i_h$ and let us denote this 
triangle by $\tau_\ell$ and its neighbor by $\tau_p$, see Figure \ref{fig:ref2}.
\begin{figure}[h]
\centering
{\psfig{figure=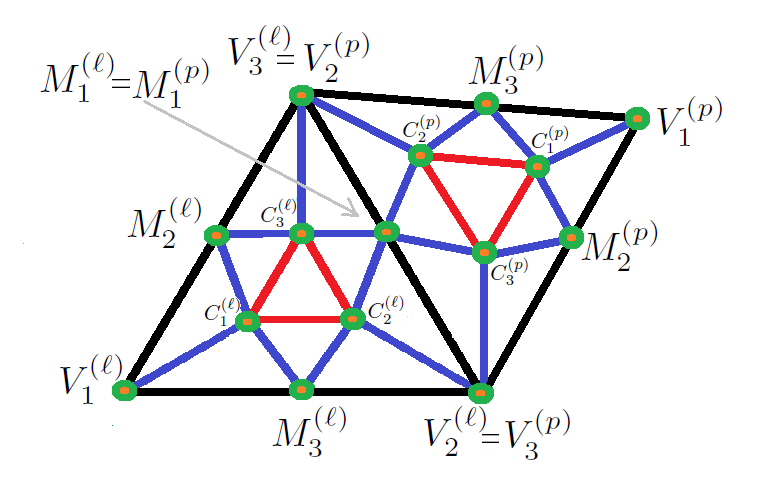,height=6cm,width=10cm,angle=0}}
\caption{Illustration of refinement of two neighboring elements.}
\label{fig:ref2}
\end{figure}

Let us denote by $\underline{\tau}$ a generic triangle of $\underline{\mathcal{T}}_h^i$. We have
\[
\int_{\tau_\ell}| \nabla(\underline{u}_i)_i|^2dx=\sum_{\underline{\tau}\subset \tau_\ell}
\int_{\underline{\tau}} | \nabla(\underline{u}_i)_i|^2dx.
\]

The sum runs over ten triangles listed in the first case of Definition \ref{definitri}. 
Let $(u_i)_i$ and $(\underline{u}_i)_i$ on $\tau$ be denoted 
by $u_i^{(\ell)}$ or $u^{(\ell)}$ and $\underline{u}_i^{(\ell)}$
or $\underline{u}^{(\ell)}$, respectively.
Note that in the triangle $\Delta^{(\ell)}:=C_1^{(\ell)}C_2^{(\ell)}C_3^{(\ell)}\subset \tau_\ell$ we have
$(u_i)_i|_{\tau_\ell}=(u_i^{\ell})=u^{(\ell)}$ and then
\[
\int_{\Delta^{(\ell)}}| \nabla \underline{u}_i^{(\ell)}|^2 dx = 
\int_{\Delta^{(\ell)}}| \nabla u_i^{(\ell)}|^2 dx=
\int_{\Delta^{(\ell)}}| \nabla u^{(\ell)}|^2 dx.
\]

Let us consider now the triangle $\Delta^{(\ell)}:=C_2^{(\ell)}C_3^{(\ell)}M_1^{(\ell)}$
where $M_1^{(\ell)}=M_1^{(p)}$. We have (see Figure \ref{fig:ref2})
\begin{eqnarray*}
I:&=&\int_{\Delta^{(\ell)}}| (\nabla ({u}_i^{(\ell)}))_i|^2 dx\leq 
C\Big\{ \Big[u^{(\ell)}( C_3^{(\ell)})-u^{(\ell)}(C_2^{(\ell)})\Big]^2+\\
&&\Big[u^{(\ell)}( C_3^{(\ell)})- 0.5\big( u^{(\ell)}(M_1^{(\ell)})+u^{(p)}(M_1^{(p)})\big)\Big]^2+\\
&&\Big[u^{(\ell)}( C_2^{(\ell)})- 0.5\big( u^{(\ell)}(M_1^{(\ell)})+u^{(p)}(M_1^{(p)})\big)\Big]^2\\
&&:= I_1+I_2+I_3.
\end{eqnarray*}
The first term above, $I_1$, it is estimated by $\|\nabla u^{(\ell)}\|_{L^2(\Delta^{(\ell)})}^2$. 
The second term, $I_2$, it is estimated as follows. We have
\begin{eqnarray}\nonumber
I_2&\leq &  C\Big\{ \big[ u^{(\ell)}(C_3^{(\ell)})-u^{(\ell)}(M_1^{(\ell)})\big]^2+
\frac{1}{2} \big[ u^{(\ell)}(M_1^{(\ell)})-u^{(p)}(M_1^{(p)})\big]^2\\
&\leq & C\Big\{ \|\nabla u^{(\ell)}\|_{L^2(\Delta^{(\ell)})}^2+
\frac{1}{h}\|u^{(\ell)}-u^{(p)}\|^2_{L^2(V_2^{(\ell)}V_3^{(\ell)})}\Big\} ,\label{A3}
\end{eqnarray}
where here $V_2^{(\ell)}V_3^{(\ell)}$ denotes the edge of $\tau$ with the end points $V_2^{(\ell)}$ and 
$V_3^{(\ell)}$. In the same way we can estimate the third term, $I_3$. Thus, 
\begin{equation}\label{Ibound}
I\leq C\left\{ \| \nabla u^{(\ell)}\|^2_{L^2(\tau_\ell)}+
\frac{1}{h}\|u^{(\ell)}-u^{(p)}\|^2_{L^2(V_1^{(\ell)}V_3^{(\ell)})}\right\}.
\end{equation}
Similarly it is possible to estimate the terms involving the triangles 
$C_1^{(\ell)}C_2^{(\ell)}M_3^{(\ell)}$ and $C_1^{(\ell)}C_3^{(\ell)}M_2^{(\ell)}$.

We now estimate the term on $\Delta^{\ell}:=C_2^{(\ell)}M_3^{(\ell)}V_2^{(\ell)}$.
We have then
\begin{eqnarray}
I:&=&\int_{\Delta^{(\ell)}}| \nabla \underline{u}_i^{(\ell)}|^2 dx\leq 
C\Big\{ \Big[u^{(\ell)}( C_2^{(\ell)})-\frac{1}{2}
\big( u^{(\ell)}(M_3^{(\ell)})+u^{(k)}(M_2^{(k)})\big)\Big]^2+\nonumber\\
&&\Big[u^{(\ell)}( C_2^{(\ell)})- \frac{1}{n_{\ell pk}}\big( u^{(\ell)}(V_2^{(\ell)})+
 u^{(p)}(V_3^{(\ell)})+\cdots +u^{(k)}(V_2^{(k)})\big)\Big]^2+\nonumber\\
&&\Big[\frac{1}{2}\big(u^{(\ell)}( M_3^{(\ell)})+u^{(k)}(M_2^{(k)})\big)-\nonumber\\
&& \frac{1}{n_{\ell pk}}
\big( u^{(\ell)}(V_2^{(\ell)})+u^{(p)}(V_3^{(\ell)})+\cdots+u^{(k)}(V_2^{(k)}))\big)\Big]^2\nonumber\\
&&:= I_1+I_2+I_3.\label{A5}
\end{eqnarray}
Here $\tau_k$ has a common edge $V_1^{(\ell)}V_2^{(\ell)}$ with $\tau_\ell$ and 
$n_{\ell p k}$ is the number of triangles of $\mathcal{T}_h^i$ with common vertex $V_2^{(\ell)}$.
The first term, $I_1$, and then second term, $I_2$ are estimated as in (\ref{A3}). 
The third term, $I_3$, it is estimated in a similar way by adding and substracting
the quantity $\big(u^{(\ell)}(V_2^{(\ell)})+u^{(k)}(V_2^{(k)})\big) $, see Figure \ref{fig:ref2}.
We proceed as above and using these estimates in (\ref{A5}) we obtain
\begin{eqnarray}
I&\leq& C\Big\{ \|\nabla u^{(\ell)}\|_{L^2(\tau_\ell)}^2+
 \|\nabla u^{(k)}\|_{L^2(\tau_k)}^2+\nonumber\\
&&\frac{1}{h}
\Big\{ \| u^{(\ell)}-u^{(k)}\|^2_{L^2(\partial\tau_\ell\cap \partial\tau_k)} +\cdots +\|u^{(\ell)}
-u^{(p)}\|_{L^2(\partial\tau_\ell\cap\partial\tau_p)}^2 \Big\} \Big\}.\label{A6}
\end{eqnarray}
In a similar way are estimated the terms over the remaining triangles of $\tau_\ell$.\\

Using the above estimates we show that 
\begin{equation}
\int_{\tau_\ell} |\nabla ( \underline{u}_i)_i|^2\leq C\Big\{ 
\sum_{\tau}\|\nabla ( u_i)_i\|_{L^2(\tau)}^2 +\frac{1}{h}
\sum_{e}\|(u_i)_i^+-(u_i)_i^-\|_{L^2( e)}^2
\Big\}\end{equation}
where the first sum runs over the elements $\tau$ which intersect $\tau_\ell$ by 
an edge and the second sum runs over edges e of $\tau$ which have a common vertex 
or edge with $\tau_\ell$.

{\bf Second case (Figure \ref{fig:ref}, upper-right picture).} We now consider the case when a vertex of $\Omega_i$ is common for two and more triangles of 
$\mathcal{T}_i^h$. Let us consider the case of two triangles, see Figure \ref{fig:A2}.  This case is estimated similar as the first case. 
\begin{figure}[htb]
\centering
{\psfig{figure=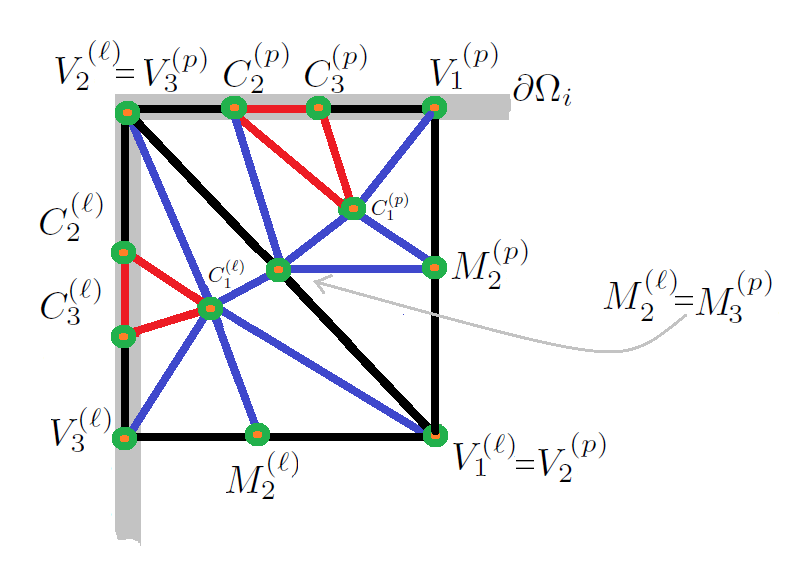,height=6cm,width=10cm,angle=0}}
\caption{Illustration of refinement of two neighboring elements.}
\label{fig:A2}
\end{figure}

{\bf Third case (Figure \ref{fig:ref}, lower-left picture).}
This case (see 
Figure \ref{fig:A3}) is also estimated similar as the first case. 

\begin{figure}[htb]
\centering
{\psfig{figure=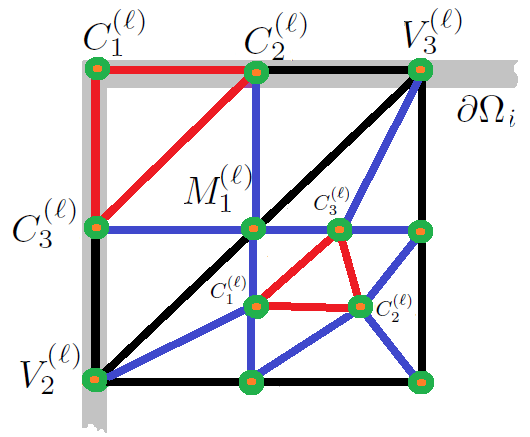,height=6cm,width=7cm,angle=0}}
\caption{Illustration of refinement of two neighboring elements.}
\label{fig:A3}
\end{figure}

Adding the above estimates for the three cases we get the estimate (\ref{A.1}) for the case 
$\underline{u}_i=\underline{I}^i_hu_i$. For the case  $\underline{u}_i=\underline{I}^{i,j}_hu_i$
we need only some minor modifications of the proof of the second case above, see Figure \ref{fig:A2}. 
This finishes the proof of the left hand side inequality of  (\ref{eq:firstinequality}).\\

We now present the proof of the left hand side of 
the result stated in Lemma \ref{underlineII}, Equation  (\ref{eq:firstinequality}). We need to show that 
there exists a constant $C$ such that
\begin{equation}\label{A.8}
a_i(({u}_i)_i,(u_i)_i)\leq C\underline{a}_i ( (\underline{u}_i)_i,(\underline{u}_i)_i).
\end{equation}
Note that, on $ \tau_\ell \in \mathcal{T}^h_i$ with 
vertices of type $C$, we have (see Figure \ref{fig:ref2})
\begin{eqnarray}
\| \nabla (u_i)_i\|^2_{L^2(\tau_\ell)}&\leq& 
C\Big\{ \big[ u^{(\ell)}(C_1^{(\ell)})-u^{(\ell)}(C_2^{(\ell)})\big]^2+\nonumber \\
&&  \big[ u^{(\ell)}(C_1^{(\ell)})-u^{(\ell)}(C_3^{(\ell)})\big]^2+
 \big[ u^{(\ell)}(C_2^{(\ell)})-u^{(\ell)}(C_3^{(\ell)})\big]^2\Big\}\nonumber\\
&\leq & C \|\nabla (I_h^{(i)} u^{(\ell)})\|_{L^2(\tau_\ell)}^2\leq 
C\| \nabla( \underline{u}_i)_i\|_{L^2(\tau_\ell)}^2.\label{A.9}
\end{eqnarray}

This is valid for all three cases considered above. Using the estimate (\ref{A.9}) we 
prove (\ref{A.8}). The proof of the equivalence   (\ref{eq:firstinequality}) is now complete.\\

We now prove the second inequality of Lemma \ref{underlineII}, 
the inequality (\ref{eq:secondinequality}). We have 
\begin{equation}\label{A10}
p_{i,\partial}(\underline{u}_i,\underline{u}_i)=\sum_{j\in \mathcal{T}^i_H} 
\sum_{e\in \mathcal{E}^{i,j}_h }\int_e \frac{\delta}{\ell_{ij}} \frac{\rho_i}{h_e}
(\underline{u}_i-\underline{u}_j)^2 dS.
\end{equation}
On the edge $e$ we have (see Figure \ref{fig:A4}))
\begin{equation}\label{A10_1}
\int_e (\underline{u}_i-\underline{u}_j)dS=\sum_{\underline{e}\subset e}
\int_{\underline{e}} (\underline{u}_i-\underline{u}_j)^2dS
\end{equation}
where $\underline{e}$ runs over the edges of $\underline{\mathcal{T}}_i^h$,   $\underline{e}\subset E_{ij}$. Note that,
on $\underline{e}=[C_2^+,C_3^+]=[C_2^-,C_3^-]$, we can write
\[
\int_{\underline{e}}(\underline{u}_i-\underline{u}_j)^2dS=
\int_{\underline{e}}(u_i-u_j)^2 dS.
\]
Additionally, on  $\underline{e}=[V_1^+,C_2^+]=[V_1^-,C_2^-]$, we have
\begin{eqnarray*}
\int_{\underline{e}}(\underline{u}_i-\underline{u}_j)^2dS&\leq& C\Big\{ 
\big[ u_i(C_2^+)-u_j(C_2^-)\big]^2+\\
&&\big[  \frac{1}{2}( u_i(C_2^+)-u_j(\widetilde{C}_2^+) ) 
-\frac{1}{2}(u_j(C_2^-)-u_j(\widetilde{C}_2^-))\big]^2
\Big\}
\end{eqnarray*}
where $\widetilde{C}_2^+$ and $\widetilde{C}_2^-$ are the nodal points on edges 
$\widetilde{e}$
of the triangles of $\mathcal{T}_h^i$ and $\mathcal{T}_h^j$ on $E_{ij}$ and $E_{ji}$ 
with common nodal points $V_1^+$ and $V_1^-$, respectively. Thus, in this case, 
\[
\int_{\underline{e}}(\underline{u}_i-\underline{u}_j)^2dS \leq  C\big\{ 
\| u_i-u_j\|^2_{L^2(e)}+\|u_i-u_j\|^2_{L^2(\tilde{e})} \big\}
\]
where $\tilde{e}\cap e=V_1^+=V_1^-$. In the case when  $V_1^+$ or $V_2^+$ are 
corners
of $\partial \Omega_i$ we do the same modification which give $\|\nabla u_i\|_{L^2(\tau)}^2$ 
on $\tau$ with vertices  $V_1^+$ or $V_2^+$. Using these in (\ref{A10_1}) and the resulting 
estimate  into (\ref{A10}) we get an estimate of the second inequality of Lemma
 \ref{underlineII} 
for the case when $\underline{u}_i=\underline{I}_h^iu_i$. The case when 
$\underline{u}_i=I_h^{i,j}u_i$ is proved similarly.\\

Now we prove the third inequality of Lemma \ref{underlineII}, 
the inequality (\ref{5.2_1}). We have that 
(\ref{A10_1}) still holds if we replace $\underline{u}_i$ by $\underline{u}_j$ and 
$u_i$ by $u_j$, respectively. Note that (see  Figure \ref{fig:A4})
\begin{eqnarray*}
\int_e(u_i-u_j)^2dS&\leq& C\Big\{  \big[ u_i(C_1^+)-u_j(C_1^-)]^2+
\big[ u_i(C_2^+)-u_j(C_2^-)]^2\Big\}\\
&\leq & C\int_{\widetilde{e}} (\underline{u}_i-\underline{u}_j)^2dx
\end{eqnarray*}
where $\widetilde{e}=(C_1^+,C_2^+)$.  Using these estimates we see that the third inequality is valid for $\underline{u}_i=\underline{I}_iu_i$. 
The case $\underline{u}_i=\underline{I}_h^{i,j}u_i$ is similar.\\

It remains only to estimate the fourth inequality, inequality (\ref{eq:fourthinequality}). It is proved as in the third inequality 
for $\underline{u}_i=\underline{I}_h^{i,j}u_i$.\\

The proof of Lemma \ref{underlineII} is complete.
\section{Proof of Lemma \ref{underlineI}}
For the first inequality, (\ref{eq:lem571}), note that 
on $\tau\in \mathcal{T}_h^{i}$ (see Figure \ref{fig:ref} upper-left picture)
\begin{eqnarray*}
\int_\tau  |\nabla (I_h^i \underline{u}_i)|^2dS&\leq& C\Big\{ 
\big[\underline{u}_i(C_1)-\underline{u}_i(C_2)\big]^2+
\big[\underline{u}_i(C_1)-\underline{u}_i(C_3)\big]^2+
\big[\underline{u}_i(C_2)-\underline{u}_i(C_3)\big]^2\\
&\leq & C\|\nabla \underline{u}_i\|^2_{L^2(C_1C_2C_3)}\leq 
C\|\nabla \underline{u}_i\|^2_{L^2(\tau)}.
\end{eqnarray*}
Summing this for $\tau\subset \overline{\Omega}_i$ we get the first inequality.\\

To prove the second inequality, (\ref{eq:lem572}), note that 
on $e\subset \partial \Omega_i$ (see Figure \ref{fig:A4})
\begin{eqnarray*}
\| I_h^i\underline{u}_i-I^i_h\underline{u}_j\|_{L^2(e)}^2&\leq& 
C\Big\{ \big[ \underline{u}_i(C_2^+)-\underline{u}_j(C_2^-)\big]^2+
\big[ \underline{u}_i(C_3^+)-\underline{u}_j(C_3^-)\big]^2\big\}\\
&\leq & C\frac{1}{h} \| \underline{u}_i-\underline{u}_j\|_{L^2(\underline{e})}^2 \quad (\mbox{ with $\underline{e}=(C_2^+,C_3^+)$)}\\
&\leq & C\frac{1}{h} \| \underline{u}_i-\underline{u}_j\|_{L^2(e)}^2.
\end{eqnarray*}
Summing this estimate over $e\subset \partial \Omega_i$ we get the second inequality.
\begin{figure}[htb]
\centering
{\psfig{figure=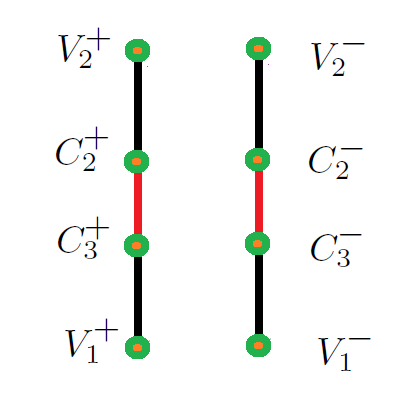,height=4cm,width=4cm,angle=0}}
\caption{Illustration of common edge refinement.}
\label{fig:A4}
\end{figure}

The equality $I_h^i\underline{u}_i=u_i$ follows from the definitions of $I^i_h$ and 
$\underline{u}_i$. 

The proof of Lemma \ref{underlineI} is now complete.

\bibliographystyle{siam}
\bibliography{references}
\end{document}